\documentclass[12pt]{amsart}
\usepackage{amssymb, amsmath, amsthm, latexsym, array,euscript}
\usepackage{fullpage}
\usepackage{verbatim}
\usepackage[scr=boondoxo,scrscaled=1.05]{mathalfa}

\usepackage{xcolor}
\definecolor{darkblue}{rgb}{0,0,.5}
\definecolor{darkgreen}{rgb}{.2,0.5,.2}
\usepackage[colorlinks=true,linkcolor=darkblue,urlcolor=violet,citecolor=magenta]{hyperref}

\numberwithin{equation}{section}

\input{cyracc.def}

\font\tencyr=wncyr10 

\def\rus{\tencyr\cyracc}

\newtheorem{thm}{Theorem}[section]
\newtheorem{lm}[thm]{Lemma}
\newtheorem{cl}[thm]{Corollary}
\newtheorem{prop}[thm]{Proposition}

\theoremstyle{remark}

\theoremstyle{definition}
\newtheorem{df}[thm]{Definition}

\newcommand {\be}{{\mathfrak b}}

\newcommand {\g}{{\mathfrak g}}
\newcommand {\h}{{\mathfrak h}}

\newcommand {\q}{{\mathfrak q}}
\newcommand {\rr}{{\mathfrak r}}

\newcommand {\te}{{\mathfrak t}}
\newcommand {\ut}{{\mathfrak u}}

\newcommand {\eus}{\EuScript}

\newcommand {\gJ}{{\eus J}}

\newcommand {\gS}{{\eus S}}
\newcommand {\gZ}{{\eus Z}}

\newcommand {\ap}{\alpha}

\newcommand {\vp}{\varphi}
\newcommand {\vth}{\vartheta}

\newcommand {\e}{\boldsymbol{e}}
\newcommand{\ard}{\rightsquigarrow}
\newcommand {\ov}{\overline}

\newcommand{\gt}{\mathfrak}

\newcommand{\GL}{{\rm GL}}

\newcommand{\Aut}{\mathsf{Aut}}
\newcommand{\Ann}{\mathrm{Ann}}
\newcommand{\ad}{\mathrm{ad}}

\newcommand{\id}{{\rm id}}
\newcommand{\ind}{{\rm ind\,}}

\newcommand{\codim}{\mathrm{codim\,}}
\newcommand{\rk}{\mathrm{rk\,}}

\newcommand{\Lie}{\mathsf{Lie\,}}

\newcommand{\oth}{\mathsf{ord}(\vartheta)}
\newcommand{\Ord}{\mathsf{ord}(\sigma)}

\newcommand {\ca}{{\mathcal A}}

\newcommand {\cz}{{\mathcal Z}}

\newcommand {\trdeg}{{\mathrm{tr.deg\,}}}

\newcommand {\mK}{{\Bbbk}}
\newcommand {\bbk}{{\Bbbk}}

\newcommand {\Z}{{\mathbb Z}}

\newcommand {\GR}[2]{{\mathsf{{#1}}}_{#2}}

\newcommand {\BZ}{{\mathbb Z}}
\newcommand {\BN}{{\mathbb N}}
\newcommand {\beq}{\begin{equation}}
\newcommand {\eeq}{\end{equation}}

\newcommand {\rg}{{\rangle}}
\renewcommand {\lg}{{\langle}}

\renewcommand{\le}{\leqslant}
\renewcommand{\ge}{\geqslant}

\newcommand {\BP}{{\mathbb P}}

\newcommand{\bb}{\boldsymbol{b}}

\newcommand {\calH}{{\mathscr{H}}}

\begin{document}
\hfill {\scriptsize May 1, 2024} 
\vskip1ex

\title[Some remarks on periodic gradings]{Some remarks on periodic gradings}
\author[O.\,Yakimova]{Oksana S.~Yakimova}
\address[O.Y.]{Institut f\"ur Mathematik, Friedrich-Schiller-Universit\"at Jena,  07737 Jena,
Deutschland}
\email{oksana.yakimova@uni-jena.de}
\thanks{This work is 
 funded by the DFG (German Research Foundation) --- project number 404144169.}
\keywords{index of Lie algebra, contraction, commutative subalgebra, symmetric invariants}
\subjclass[2010]{17B63, 14L30, 17B08, 17B20, 22E46}
\begin{abstract}
Let $\q$ be a finite-dimensional Lie algebra, $\vth\in {\sf Aut}(\q)$ a finite order automorphism, and $\q_0$ the subalgebra of fixed points of $\vth$.  Using $\vth$ one can construct a pencil $\mathcal P$ of
compatible Poisson brackets on $\gS(\q)$, and thereby a `large' Poisson-commutative subalgebra 
$\gZ(\q,\vth)$ of $\gS(\q)^{\q_0}$. In this article, we study one particular bracket 
$\{\,\,,\,\}_{\infty}\in\mathcal P$ and the related Poisson centre $\cz_\infty$. 
It is shown that $\cz_\infty$ is a polynomial ring, if $\q$ is reductive.
\end{abstract}
\maketitle

\tableofcontents
\section*{Introduction}
\label{sect:intro}

\noindent
The ground field $\mK$ is algebraically closed and $\mathsf{char}(\mK)=0$. 
Let $\q=(\q, [\,\,,\,])$ be a finite-dimensional algebraic Lie algebra, i.e., $\q=\Lie Q$, where $Q$ is a 
connected affine algebraic group. The dual space $\q^*$ is a Poisson variety, i.e., the 
algebra of polynomial functions on $\q^*$, $\bbk[\q^*]\simeq \gS(\q)$, is equipped with the Lie--Poisson 
bracket $\{\,\,,\,\}$. Here $\{x,y\}=[x,y]$ for $x,y\in\q$. 
Poisson-commutative subalgebras of $\bbk[\q^*]$ are  important tools for the study of geometry of
the coadjoint action of $Q$ and representation theory of $\q$. 

\subsection{}    \label{subs:old}
There is a well-known method, {\it the Lenard--Magri scheme}, for constructing ``large" 
Poisson-commutative subalgebras of $\bbk[\q^*]$, which is related to {\it compatible Poisson brackets}, 
see e.g.~\cite{GZ}.  
Two Poisson brackets $\{\,\,,\,\}'$ and $\{\,\,,\,\}''$ are said to be {\it compatible}, if any linear combination 
$\{\,\,,\,\}_{a,b}:=a\{\,\,,\,\}'+b\{\,\,,\,\}''$ with $a,b\in\bbk$ is a Poisson bracket. Then one defines a certain 
dense open subset $\Omega_{\sf reg}\subset\bbk^2$ that corresponds to the {\it regular} brackets in the pencil $\mathcal P=\{ \{\,\,,\,\}_{a,b}\mid (a,b)\in\bbk^2\}$.  
Let $\cz_{a,b}\subset \gS(\q)$ denote the Poisson centre of  $(\gS(\q), \{\,\,,\,\}_{a,b})$. Then
the subalgebra $\gZ\subset\gS(\q)$ generated by  $\cz_{a,b}$
with $(a,b)\in\Omega_{\sf reg}$ is Poisson-commutative w.r.t.~$\{\,\,,\,\}'$ and $\{\,\,,\,\}''$. 
An obvious first step is to take the initial Lie--Poisson 
bracket $\{\,\,,\,\}$ as $\{\,\,,\,\}'$. The rest depends on a clever choice of $\{\,\,,\,\}''$.

\subsection{}    \label{subs:old2}
Let $\vth$ be an automorphism of $\q$ of finite order $m$. Then 
$\q$ is equipped with
a $\BZ_m$-grading $\gt q=\bigoplus_{i=0}^{m-1}\gt q_i$, where each $\q_i$ is an eigenspace of $\vth$. 
One can naturally construct a compatible Poisson bracket
$\{\,\,,\,\}''$ associated with the grading~\cite{OY,fo}.  In this case, all Poisson brackets in $\mathcal P$
are linear and there are two lines $l_1,l_2\subset \bbk^2$ such that
$\Omega=\bbk^2\setminus (l_1\cup l_2)\subset \Omega_{\sf reg}$ and the Lie algebras corresponding to
$(a,b)\in \Omega$ are isomorphic to $\q$. The lines $l_1$ and $l_2$ give rise to new Lie algebras, denoted $\q_{(0)}$ and $\q_{(\infty)}$. These new algebras are different {\it contractions} of $\q$. 
A definition and basic properties of contractions are discussed in Section~\ref{subs:contr-&-inv}. 
Let 
$\ind\q$ denote the {\it index\/} of $\q$ (see Section~\ref{subs:coadj}). Then we have $\ind\q\le\ind\q_{(t)}$ for 
$t\in\{0,\infty\}$. Let $\gS(\q)^{\q}\simeq\bbk[\q^*]^Q$ be the ring of {\it symmetric invariants} of $\q$, i.e., the Poisson centre of $\gS(\q)$. 
Let further $\gZ=\gZ(\q,\vth)$ be the Poisson-commutative subalgebra associated with $\mathcal P$. Many features of 
$\gZ$ depend on the 
properties of $\q_{(0)}$ and $\q_{(\infty)}$.  

In \cite{fo,MZ23}, we have studied the ring $\cz_0=\gS(\q_{(0)})^{\q_{(0)}}$ in case $\q=\g$ is reductive. 
For some $\g$ such that $\ind\g_{(0)}=\rk\g$, this is a polynomial ring with $\rk\g$ generators. However, there are exceptions
even if $m=2$ \cite{Y-imrn}.   
Partial results were obtained for $\cz_\infty=\gS(\q_{(\infty)})^{\q_{(\infty)}}$. Namely, if $\g$ is reductive and 
$\vth$ is an inner automorphism, then $\cz_\infty=\gS(\g_0)$ \cite{fo}. 

In Section~\ref{sect:prelim}, we collect basic facts on the coadjoint action and symmetric invariants. 
Explicit descriptions of the algebras $\q_{(0)}$ and $\q_{(\infty)}$ are presented in Section~\ref{subs:periodic}. 

\subsection{}    \label{subs:new}
If $\ind\q_{(\infty)}>\ind\q$, then $\cz_\infty$ does not have to be Poisson-commutative. 
Our first result states that $\{\cz_\infty^{\q_0},\cz_\infty^{\q_0}\}=0$, if $\q_0^*$ contains a regular in  $\q^*$ element. 
Furthermore, under that assumption, the algebra $\mathsf{alg}\lg \gZ,\cz_\infty^{\q_0}\rg$,
generated by $\gZ$ and $\cz_\infty^{\q_0}$, 
 is still Poisson-commutative,
see Theorem~\ref{inf-inv}. Both statements have applications related to the current algebra $\gt q[t]$. Namely, one can construct a large Poisson-commutative subalgebra of $\gS(\q[t]^{\vth})$ following the ideas of 
\cite[Sect.\,8]{fo}. However, our new approach  works for several non-reductive Lie algebras and does not require  the assumption that $\ind\g_{(0)}=\rk\gt g$, which is imposed in \cite{fo}. The construction will appear in a forthcoming paper.

Section~\ref{red} contains a brief summary of Kac's classification of finite order automorphisms for a semisimple $\g$  \cite{k69}. 
In particular, we  describe 
a relation between the roots of $\g$ and of $\g^\sigma$, where $\sigma$ is a diagram   automorphism of $\g$. 
Then in Section~\ref{sect:3}, we state main results of \cite{fo} on generators of $\gZ(\g,\vth)$. 
Many crucial properties of $\gZ(\g,\vth)$, including its transcendence degree, depend on the equality  
$\ind\g_{(0)}=\rk\g$, see e.g. \cite[Thm\,3.10]{fo} or Theorem~\ref{free-main} here.  
Conjecture 3.1 in \cite{MZ23} states that $\ind\g_{(0)}=\rk\g$ for all $\g$ and all $\vth$. 
In Section~\ref{sec-new}, we recollect several instances, where the equality holds, and provide a few new positive examples,
see Theorem~\ref{ind-au}.  

In the more favourable reductive case, we prove that $\cz_\infty$ and $\cz_\infty^{\g_0}$ are always polynomial rings and describe their generators explicitly, see Section~\ref{sect:infty-g}. 
We consider also the non-reductive Lie algebra $\tilde\g=\g_0\ltimes\g_{(\infty)}$ and its coadjoint representation.  
It is shown that $\ind\tilde\g=\rk\g+\rk\g_0$ and that $\tilde\g$ has the  {\sl codim}--$2$ property, see Theorem~\ref{inf-c2}. 
Then by Theorem~\ref{ggs-inf}, $\gS(\tilde\g)^{\tilde\g}$ is a polynomial ring with $\ind\tilde\g$ generators.
These generators are described explicitly.  

Non-reductive Lie algebras $\q$ such that $\gS(\q)^{\q}$ is a polynomial ring with $\ind\q$ generators attract a lot of attention, see e.g.  
\cite{j,ppy,p07,p09,contr,MZ,FP}.  A quest for this type of algebras continues. Many  examples that are found so far  are related to particular simple Lie algebras.
For instance, assertions of \cite{j} and \cite{ppy} hold in full generality only for $\gt{sl}_n$ and $\gt{sp}_{2n}$.
Note that our results on $\g_{(\infty)}$ and $\tilde\g$ are independent of the type of $\g$ and apply to all finite order automorphisms. 
We prove also that $\tilde\g$ is a Lie algebra of {\it Kostant type} in the terminology of \cite{contr}. 

The Lie algebra $\tilde\g$ is {\it quadratic}, i.e., there is a $\tilde\g$-invariant non-degenerate symmetric bilinear form 
on $\tilde\g$. This implies that the adjoint and coadjoint representations of $\tilde\g$ are isomorphic, see e.g. \cite[Sect.\,1.1]{p09}. In \cite[Sect.\,4]{p09}, invariants of the adjoint action of $\tilde\g$ are studied. In the notation of that 
paper, $\tilde\g=\g\langle m+1\rangle_0$.  A more general object $\g\langle nm+1\rangle_0$ of \cite{p09} can be  also interpreted in our context. 

Suppose that $\g=\h^{\oplus n}$ is a sum of $n$ copies of a reductive Lie algebra $\h$ and $\vth$ is a composition of a finite order automorphism $\tilde\vth\in\Aut(\h)$ and a cyclic permutation of the 
summands. Formally we have 
\begin{equation}\label{tildet}
\vth\left((x_1,x_2,\ldots,x_n)\right)=(x_n,\tilde\vth(x_1),x_2,\ldots,x_{n-1})
\ \ 
\text{ for } 
\ \ (x_1,x_2,\ldots,x_n)\in\gt h\oplus\gt h\oplus\ldots\oplus\gt h. 
\end{equation} 
Then $\tilde\g$ associated with $\vth$ is equal to $\gt h\langle nm+1\rangle_0$. 
In \cite[Thm\,4.1{\sf (ii)}]{p09}, it is shown that $\bbk[\g\langle nm+1\rangle_0]^{\g\langle nm+1\rangle_0}$ is a polynomial algebra of Krull dimension $n\rk\g+\rk\g_0$, if $\g_0$ contains a regular nilpotent element of $\g$. 
Here we have no assumptions on $\g_0$, i.e., we show that \cite[Thm\,4.1{\sf (ii)}]{p09} holds for all $\g$ and $\vth$. 

Our general reference for semisimple Lie groups and algebras is \cite{t41}.

\section{Preliminaries on Poisson brackets and polynomial contractions}
\label{sect:prelim}

\noindent
Let $Q$\/ be a connected affine algebraic group with Lie algebra $\q$. The  symmetric algebra of 
$\q$ over $\mK$ is $\BN_0$-graded, i.e., $\gS(\q)=\bigoplus_{i\ge 0}\gS^i(\q)$. It is identified with the 
algebra of polynomial functions on the dual 
space $\q^*$.   
\subsection{The coadjoint representation}
\label{subs:coadj}
 The group $Q$ acts on $\q^*$ via the coadjoint 
representation and then $\ad^*\!: \q\to \mathrm{GL}(\q^*)$ is the {\it coadjoint representation\/} of $\q$. 
The algebra 
of $Q$-in\-va\-riant polynomial functions on $\q^*$ is denoted by $\gS(\q)^{Q}$ or $\mK[\q^*]^Q$.
Write $\mK(\q^*)^Q$ for the field of $Q$-invariant rational functions on $\q^*$.
\\ \indent
Let $\q^\xi=\{x\in\q\mid \ad^*(x){\cdot}\xi=0\}$ be the {\it stabiliser\/} in $\q$ of $\xi\in\q^*$. The 
{\it index of\/} $\q$, $\ind\q$, is the minimal codimension of $Q$-orbits in $\q^*$. Equivalently,
$\ind\q=\min_{\xi\in\q^*} \dim \q^\xi$. By the Rosenlicht theorem (see~\cite[IV.2]{spr}), one also has 
$\ind\q=\trdeg\mK(\q^*)^Q$. 
Set $\bb(\q)=(\dim\q+\ind\q)/2$. 
Since the $Q$-orbits in $\q^*$ are even-dimensional, $\bb(\q)$ is an integer. If $\q$ is reductive, then  
$\ind\q=\rk\q$ and $\bb(\q)$ equals the dimension of a Borel subalgebra.  

The Lie--Poisson bracket on $\gS(\q)$ is defined on $\gS^1(\q)=\q$ by $\{x,y\}:=[x,y]$. It is then extended to higher degrees via the Leibniz rule. Hence $\gS(\q)$ has the usual associative-commutative structure and additional Poisson structure.  
Whenever we refer to {\sl subalgebras\/} of $\gS(\q)$, we always mean the associative-commutative structure.
Then  a subalgebra $\ca\subset \gS(\q)$ is said to be {\it Poisson-commutative}, 
if $\{H,F\}=0$ for all $H,F\in\ca$. It is well known that if $\ca$ is Poisson-commutative, then $\trdeg\ca\le \bb(\q)$, see e.g.~\cite[0.2]{vi90}. More generally, suppose that 
$\h\subset\q$ is a Lie subalgebra 
and $\ca\subset  \gS(\q)^\h$ is Poisson-commutative. Then
$\trdeg\ca\le \bb(\q)-\bb(\h)+\ind\h$, see~\cite[Prop.\,1.1]{m-y}.

The {\it centre\/} of the Poisson algebra $(\gS(\q), \{\,\,,\,\})$ is 
\[
    \cz(\q):=\{H\in \gS(\q)\mid \{H,F\}=0 \ \ \forall F\in\gS(\q)\} =
   \gS(\q)^\q=\bbk[\q^*]^\q =\bbk[\q^*]^Q.
\] 
Since the quotient field of
$\mK[\q^*]^Q$ is contained in $\bbk(\q^*)^Q$, we deduce from the Rosenlicht theorem that
\beq    \label{eq:neravenstvo-ind}
    \trdeg (\gS(\q)^\q)\le \ind\q .
\eeq
The set of {\it regular\/} elements of $\q^*$ is 
\beq       \label{eq:regul-set}
    \q^*_{\sf reg}=\{\eta\in\q^*\mid \dim \q^\eta=\ind\q\}=\{\eta\in\q^*\mid \dim Q{\cdot}\eta \ \text{ is maximal}\} .
\eeq
It is a dense open subset of $\q^*$.
Set $\q^*_{\sf sing}=\q^*\setminus \q^*_{\sf reg}$.
We say that $\q$ has the {\sl codim}--$n$ property if $\codim \q^*_{\sf sing}\ge n$. 
The {\sl codim}--$2$ property  is going to be most important for us. 

For $\gamma\in\q^*$, let $\hat\gamma$ be the skew-symmetric bilinear form on $\q$ defined by 
$\hat\gamma(\xi,\eta)=\gamma([\xi,\eta])$ for $\xi,\eta\in\q$. It follows that
$\ker\hat\gamma=\q^\gamma$. The $2$-form $\hat\gamma$ is related to 
the {\it Poisson tensor (bivector)} $\pi$ of the Lie--Poisson bracket $\{\,\,,\,\}$ as follows.

Let $\textsl{d}H$ denote the differential of $H\in \gS(\q)=\bbk[\q^*]$. Then 
$\pi$ is defined by the formula
$\pi(\textsl{d}H\wedge \textsl{d}F)=\{H,F\}$ for $H,F\in\gS(\q)$. Then 
$\pi(\gamma)(\textsl{d}_\gamma H\wedge \textsl{d}_\gamma F)=\{H,F\}(\gamma)$ and therefore
$\hat\gamma=\pi(\gamma)$.
In this terms, $\ind\q=\dim\q-\rk\pi$, where $\rk\pi=\max_{\gamma\in\q^*}\rk\pi(\gamma)$. 

For a subalgebra $A\subset\gS(\q)$ and $\gamma\in\gt q^*$, set 
$\textsl{d}_\gamma A=\left<\textsl{d}_\gamma F \mid F\in A \right>_{\mK}$. By the definition of 
$\gS(\q)^{\q}$, we have 
\begin{equation} \label{incl}
\textsl{d}_\gamma \gS(\q)^{\q}\subset \ker\pi(\gamma)
\end{equation}
 for each $\gamma\in\q^*$. 

\subsection{Contractions and invariants}
\label{subs:contr-&-inv} 
We refer to \cite[Ch.\,7,\,\S\,2]{t41} for basic facts on contractions of Lie algebras.
In this article, we consider contractions of the following form. Let $\bbk^\star=\bbk\setminus\{0\}$ be 
the multiplicative group of $\bbk$ and 
$\vp: \bbk^\star\to \GL(\q)$, $s\mapsto \vp_s$, a polynomial representation. That is, 
the matrix entries of $\vp_{s}:\q\to \q$ are polynomials in $s$ w.r.t. some (any) basis of $\q$.
Define a new Lie algebra structure on the vector space $\q$ and associated Lie--Poisson bracket by 
\beq       \label{eq:fi_s}
      [x, y]_{(s)}=\{x,y\}_{(s)}:=\vp_s^{-1}[\vp_s( x), \vp_s( y)], \ x,y \in \q, \ s\in\bbk^\star.
\eeq
All the algebras $(\q,[\,\,,\,]_{(s)})$  
are isomorphic and  $(\q,[\,\,,\,]_{(1)})$ is the initial Lie algebra $\q$. The induced $\bbk^\star$-action on the variety of structure constants is not necessarily 
polynomial, i.e., \ $\lim_{s\to 0}[x, y]_{(s)}$ may not exist for all $x,y\in\q$. Whenever such a limit exists, 
we obtain a new linear Poisson bracket, denoted $\{\,\,,\,\}_0$, and thereby a new Lie algebra $\q_{(0)}$, 
which is said to be a {\it contraction\/} of $\q$. If we wish to stress that this construction is determined 
by $\vp$, then we write $\{x, y\}_{(\vp,s)}$ for the bracket in~\eqref{eq:fi_s} and say that $\q_{(0)}=\q_{(0,\vp)}$ is the 
$\vp$-{\it contraction\/} of $\q$ or is the {\it zero limit of $\q$ w.r.t.}~$\vp$.  
A criterion for the existence of $\q_{(0)}$
can be given in terms of Lie brackets of the $\vp$-eigenspaces in $\q$, see~\cite[Sect.\,4]{Y-imrn}. 
We identify all algebras $\q_{(s)}$ and 
$\q_{(0)}$ as vector spaces. The semi-continuity of index implies that $\ind\q_{(0)}\ge \ind\q$.

The map $\vp_s$, $s\in\bbk^\star$, is naturally extended to an invertible transformation of 
$\gS^j(\q)$, which we also denote by $\vp_s$. The resulting graded map 
$\vp_s:\gS(\q)\to\gS(\q)$ is nothing but the comorphism associated with $s\in\bbk^\star$ and
the dual representation
$\vp^*:\bbk^\star\to \GL(\q^*)$.
Since $\gS^j(\q)$ has a basis that consists of $\vp(\bbk^\star)$-eigenvectors, any $F\in\gS^j(\q)$  
can be written as $F=\sum_{i\ge 0}F_i$, 
where the sum is finite and $\vp_s(F_i)=s^iF_i\in\gS^j(\q)$. Let $F^\bullet$ denote the non-zero component $F_i$ with maximal $i$.

\begin{prop}[{\cite[Lemma~3.3]{contr}}]     \label{prop:bullet}
If $F\in\cz(\q)$ and $\q_{(0)}$ exists, then $F^\bullet\in \cz(\q_{(0)})$. 
\end{prop}

\subsection{Periodic gradings of Lie algebras and related compatible brackets }      
\label{subs:periodic}
Let $\vartheta\in\Aut(\q)$ be a Lie algebra automorphism of finite order $m\ge 2$ and $\zeta=\sqrt[m]1$ 
a primitive root of unity. Write also $\oth$ for the order of $\vartheta$.
If $\q_i$ is the $\zeta^i$-eigenspace of $\vartheta$, $i\in \BZ_m$, then the direct sum
$\q=\bigoplus_{i\in \BZ_m}\q_i$ is a {\it periodic grading\/} or $\BZ_m$-{\it grading\/} of 
$\q$. The latter means that $[\q_i,\q_j]\subset \q_{i+j}$ for all $i,j\in \BZ_m$. Here $\q_0=\q^\vartheta$ 
is the fixed-point subalgebra for $\vartheta$ and each $\q_i$ is a $\q_0$-module.

We choose $\{0,1,\dots, m-1\}\subset\BZ$ as a fixed set of representatives for $\BZ_m=\BZ/m\BZ$. 
Under this convention, we have
$\q=\q_0\oplus\q_1\oplus\ldots\oplus\q_{m-1}$ and
\beq   \label{eq:Z_m}
[\q_i,\q_j]\subset \begin{cases}  \q_{i+j}, &\text{ if } \ i+j\le m{-}1, \\
 \q_{i+j-m}, &\text{ if } \ i+j\ge m. \end{cases}
\eeq
This is needed below, when we consider $\BZ$-graded contractions of $\q$ associated with $\vartheta$. 

The presence of $\vartheta$ allows us to split the Lie--Poisson bracket on $\q^*$ into a sum of two compatible Poisson brackets.
Consider the polynomial representation $\vp\!:\bbk^\star\to {\rm GL}(\q)$ 
such that $\vp_s(x)=s^j x$ for $x\in \q_j$.
As in Section~\ref{subs:contr-&-inv}, this defines a family of linear Poisson brackets on $\gS(\q)$ 
parametrised by $s\in\bbk^\star$, see~\eqref{eq:fi_s}.

Below we outline some results of \cite{fo}:
\begin{itemize}
\item[\sf (i)]  \ There is a limit\/ $\lim_{s\to 0}  \{x,y\}_{(s)}=:\{x,y\}_{0}$, which is a linear 
Poisson bracket on $\gS(\q)$.
\item[\sf (ii)] \ The difference $\{\,\,,\,\}-\{\,\,,\,\}_{0}=:\{\,\,,\,\}_{\infty}$ is a linear Poisson bracket on 
$\gS(\q)$, which 
 is obtained as the zero limit w.r.t. the polynomial  
representation $\psi\!:\bbk^\star \to {\rm GL}(\q)$ such that $\psi_s=s^m{\cdot}\vp_{s^{-1}}=
s^m{\cdot}\vp_s^{-1}$, $s\in\bbk^\star$. In other words, 
$\{\,\,,\,\}_{\infty}=\lim_{s\to 0}\{\,\,,\,\}_{(\psi,s)}$.
\item[\sf (iii)] \ For any $\vartheta\in \Aut(\q)$ of finite order,
the Poisson brackets $\{\,\,,\,\}_{0}$ and $\{\,\,,\,\}_{\infty}$ are compatible, and the corresponding pencil
contains the initial Lie--Poisson bracket.
\end{itemize}

Set
\[
   \{\,\,,\,\}_{t}  = \{\,\,,\,\}_{0}+ t\{\,\,,\,\}_{\infty} , 
\]
where $t\in \BP:=\bbk\cup\{\infty\}$ and the value $t=\infty$ corresponds to the bracket
$\{\,\,,\,\}_{\infty}$. Let $\q_{(t)}$ stand for the Lie algebra corresponding to $\{\,\,,\,\}_{t}$.
All these Lie algebras have the same underlying vector space.

\begin{prop}[{\cite[Prop.~2.3]{fo}}]   \label{prop:graded-limits}
The Lie algebras $\q_{(0)}$ and $\q_{(\infty)}$ are $\BN_0$-graded. More precisely, if\/ $\rr[i]$ stands for the component of grade $i\in\BN_0$ in an $\BN_0$-graded Lie algebra $\rr$, then 
\[
    \q_{(0)}[i]=\begin{cases} \q_i \  & \text{ for } i=0,1,\dots,m{-}1 \\
       0 & \text{ otherwise}   \end{cases}, \ 
    \q_{(\infty)}[i]=\begin{cases} \q_{m-i} \  & \text{ for } i=1,2,\dots,m  \\
      0 & \text{ otherwise}   \end{cases}.
\]
In particular, $\q_{(\infty)}$ is nilpotent and the subspace $\q_0$, which is the 
highest grade component of $\q_{(\infty)}$, belongs to the centre of $\q_{(\infty)}$. \qed
\end{prop}

Suppose that $t\in\bbk^\star$. Then $\{\,\,,\,\}_{t}= \{\,\,,\,\}_{(s)}$, where $s^m=t$. 
The algebras $\q_{(t)}$ 
with $t\in\bbk^\star$ are isomorphic and they have one and the same index. We say that $t\in\BP$ is 
{\it regular} if $\ind\q_{(t)}=\ind\q$ and write $\BP_{\sf reg}$ for the set of regular values.
Then $\BP_{\sf sing}:=\BP\setminus \BP_{\sf reg}\subset \{0,\infty\}$  is the set of singular values.

Let $Q_0\subset Q$ be the connected subgroup of $Q$ with 
$\Lie Q_0=\q_0$. It is easy to see that $Q_0$ is an algebraic group. Hence there are  
connected algebraic
groups  $Q_{(t)}$ such that $\q_{(t)}=\Lie Q_{(t)}$  for each $t\in\BP$. 

Let $\cz_t$ be the centre of the Poisson algebra $(\gS(\q),\{\,\,,\,\}_t)$.
In particular, 
$\cz_1=\gS(\q)^{\q}$. If $t=s^m\in\bbk^\star$, then $\cz_t=\vp_s^{-1}(\cz_1)$.  
By Eq.~\eqref{eq:neravenstvo-ind}, we have $\trdeg\cz_t\le\ind\q_{(t)}$. 
In \cite{fo}, we have studied the 
subalgebra $\gZ\subset\gS(\q)$ generated by the centres $\cz_t$ with $t\in\BP_{\sf reg}$, i.e.,
\[
     \gZ=\gZ(\gt q,\vth)=\mathsf{alg}\langle\cz_t \mid t\in\BP_{\sf reg}\rangle .
\]
By a general property of compatible brackets, the algebra $\gZ$
is Poisson-commutative w.r.t. {\bf all} brackets $\{\,\,,\,\}_t$ with $t\in\BP$, cf. \cite[Sect.~2]{OY}. 
Note that the Lie subalgebra $\q_0\subset\q=\q_{(1)}$ is also the {\bf same} Lie subalgebra in any $\q_{(t)}$ with 
$t\ne\infty$ (cf. Proposition~\ref{prop:graded-limits} for $\q_{(0)}$). Therefore, 
\begin{equation}         \label{inclz}
      \cz_t\subset\gS(\q)^{\q_0} \ \text{ for } \,t\ne\infty.
\end{equation}

{\bf Convention.} 
We think of $\q^*$ as the dual space for any Lie algebra $\q_{(t)}$
and sometimes omit the subscript `$(t)$' in $\q_{(t)}^*$. However, if $\xi\in\q^*$,
then the stabiliser of $\xi$ with respect to the coadjoint representation
of $\q_{(t)}$ is denoted by $\q_{(t)}^\xi$. Set  $\q^*_{\infty,\sf reg}:=(\q_{(\infty)}^*)_{\sf reg}$. 

\subsection{Good generating systems} \label{sec-ggs}
Let $\g$ be a reductive Lie algebra. Consider a contraction $\g_{(0)}$ of $\g$ given by 
$\{\vp_s \mid s\in\bbk^\star\}$. 
Let $\{x_1,\ldots,x_{\dim\g}\}$ be basis of $\g$ and $\omega=x_1\wedge\ldots\wedge x_{\dim\g}$ a volume 
form. Then $\varphi(\omega)=s^D \omega$, where $D$ is a non-negative integer. 
We set $D_\vp:=D$. 

The Poisson centre $\gS(\g)^\g$ is a polynomial algebra of Krull dimension $l=\rk\g$ and 
$\ind\g=l$. Hence one has now the equality in Eq.~\eqref{eq:neravenstvo-ind}.
Let $\{H_1,\dots,H_l\}$  be a set of homogeneous algebraically independent generators of
$\gS(\g)^\g$ and $d_i=\deg H_i$. Then $\sum_{i=1}^l d_i=\bb(\g)$. As above, each $H_j$ decomposes as 
$H_j=\sum_{i\ge 0} H_{j,i}$, where $\varphi_s(H_j)=\sum_{i\ge 0} s^i H_{j,i}$. The polynomials
$H_{j,i}$ are called {bi-homogeneous components} of $H_j$. By definition, the $\vp$-{\it degree\/} of
$H_{j,i}$ is $i$, also denoted by $\deg_\vp H_{j,i}$.
Then $H_j^\bullet$ is the non-zero bi-homogeneous component of $H_j$ with
maximal $\vp$-degree. We set $\deg_{\vp}\! H_j=\deg_{\vp}\! H_j^\bullet$ and 
$d_j^\bullet=\deg_{\vp}\! H_j^\bullet$.

\begin{df}
Let us say that $H_1,\dots,H_l$ is a {\it good generating system} in $\gS(\g)^\g$
({\sf g.g.s.}\/ {\it for short}) for $\vp$, if $H_1^\bullet,\dots,H_l^\bullet$ are
algebraically independent. Then we also say that $\vp$ {\it admits\/} a {g.g.s.}
\end{df}
\noindent
The property of being `good' really depends on a generating system. 
The importance of {\sf g.g.s.} is manifestly seen in the following  result.

\begin{thm}[{\cite[Theorem\,3.8]{contr}}]    \label{thm:kot14}
Let $H_1,\dots,H_l$ be an arbitrary set of homogeneous algebraically independent generators of\/ $\gS(\g)^\g$. Then
\begin{itemize}
\item[\sf (i)] \ $\sum_{j=1}^l \deg_{\vp}\! H_j\ge D_\vp$\,;
\item[\sf (ii)] \  $H_1,\dots,H_l$ is a {\sf g.g.s.} if and only if\/ $\sum_{j=1}^l \deg_{\vp}\! H_j=D_\vp$\,;
\item[\sf (iii)] \  if\/ $\g_{(0)}$ has the {\it codim}--$2$ property, $\ind\g_{(0)}=l$, and  $H_1,\dots,H_l$ is a {\sf g.g.s.}, then
$\cz_0=\gS(\g_{(0)})^{\g_{(0)}}$ is a polynomial algebra freely generated by 
$H_1^\bullet,\dots,H_l^\bullet$ and 
\[
\phantom{1234156775341234789}
\{\xi\in\g^*\mid \textsl{d}_\xi H_1^\bullet\wedge\ldots\wedge\textsl{d}_\xi H_l^\bullet=0\}=(\g_{(0)})^*_{\sf sing}. 
\phantom{1234516775413234789}\Box
\]
\end{itemize}
\end{thm}

Let $F_1,\ldots,F_N\in\mK[x_i\mid 1\le i\le n]=\mK[\mathbb A^n]$ be algebraically independent 
homogeneous  polynomials. Set 
\[
\gJ(F_1,\ldots,F_N):=\{x\in\mathbb A^n \mid \textsl{d}_x F_1\wedge\ldots\wedge \textsl{d}_x F_N=0\}.
\]
Then $\gJ(F_1,\ldots,F_N)$ is a proper closed subset of $\mathbb A^n$. 
An open subset of $\mathbb A^n$ is said to be {\it big}, if its complement does not contain divisors. 
Thereby the differentials $\textsl{d}F_i$ with $1\le i\le N$ are  linearly independent on a  big open subset
if and only if $\dim \gJ(F_1,\ldots,F_N)\le n-2$. 

By the {\it Kostant regularity criterion\/} for
$\g$, $\gJ(H_1,\ldots,H_l)=\g^*_{\sf sing}$, see~\cite[Theorem~9]{ko63}. 
By~\cite{ko63}, $\g$ has the {\sl codim}--$3$ property, i.e., $\dim \g^*_{\sf sing}\le \dim\g-3$.

Part {\sf (iii)} of Theorem~\ref{thm:kot14} states that the Kostant regularity criterion holds for $\g_{(0)}$. 
One of the ingredients of the proof, which will be used  in this paper as well, is the following statement.  

\begin{thm}[{\cite[Theorem~1.1]{ppy}}]       \label{ppy-max}  
Let $F_1,\ldots,F_N$ be as above. 
If $\dim\gJ(F_1,\ldots,F_N)\le n-2$, then ${\mathcal F}=\mK[F_j\mid 1\le j\le N]$ is an 
algebraically closed subalgebra of  $\mK[\mathbb A^n]$, i.e., if $H\in \mK[\mathbb A^n]$ is 
algebraic over the field\/ $\mK(F_1,\ldots,F_N)$, then
$H\in \mathcal F$. 
\end{thm}

\section{Properties of $\q_{(\infty)}$ and of the Poisson centre $\cz_\infty\subset\gS(\q_{(\infty)})$}
\label{sect:infty}

\noindent
By Proposition~\ref{prop:graded-limits}, $\q_{(\infty)}$ is a nilpotent  $\BN$-graded Lie 
algebra. 
Recall also that the subspace $\q_0$ belongs to the centre of $\q_{(\infty)}$. 
Let $\pi_t$ be the Poisson tensor of $\gt q_{(t)}$.  We identify $\q_0^*$ with the annihilator 
$\Ann(\bigoplus_{i=1}^{m-1}\q_i)\subset\q^*$ and regard it as a subspace of $\q^*$. 

For any $\xi\in\gt q_0^*\subset\q^*$, we have 
\beq \label{q0-inf}
\q^\xi_{(\infty)}=\q_0\oplus\bigoplus_{j\ge 1}\q_j^\xi,  \ \ \text{ where } \ \ 
\q_j^\xi=\{y\in\q_j \mid \xi([y,\q_{m-j}])=0\}.
\eeq
Furthermore, $\q^\xi=\q_0^\xi\oplus\bigoplus_{j\ge 1}\q_j^\xi$. If this $\xi$ is regular in $\q^*$, then  
$\xi\in(\q_0^*)_{\sf reg}$. 

\begin{thm}            \label{thm:ind-inf}
Suppose that $\gt q_0^*\cap\gt q^*_{\sf reg}\ne\varnothing$. Then 
$\ind\q_{(\infty)}= \dim\q_0+\ind\q-\ind\q_0$.
\end{thm}
\begin{proof}
{\bf (1)} Take any $\xi\in \gt q_0^*\cap\gt q^*_{\sf reg}$. Then $\q_{(\infty)}^\xi=\q^\xi+\q_0$.
Furthermore $\q^\xi\cap\q_0 = \q_0^{\xi}$.   As we have explained above,  $\xi\in (\q_0^*)_{\sf reg}$. Thereby 
$\dim\q_0^\xi=\ind\q_0$
 and hence 
$\dim \q_{(\infty)}^\xi=\ind\q+\dim\q_0-\ind\q_0$. This leads to $\ind\q_{(\infty)}\le \ind\q+\dim\q_0-\ind\q_0$. 

{\bf (2)} \ Let us prove the opposite inequality.
Take any $\xi\in\q^*_{\sf reg}$ such that $\bar\xi=\xi|_{\q_0}\in(\q_0^*)_{\sf reg}$. Note that there is a non-empty open subset consisting of suitable elements. For all but finitely many $t$, we have $\dim\q_{(t)}^\xi=\ind\q$. Hence
$\gt v=\lim\limits_{t\to\infty} \q_{(t)}^\xi$ is a well-defined subspace of $\q$ of dimension $\ind\q$.  
If $t\ne \infty$, then $\pi_t(\xi)|_{\q_0{\times}\q}=\pi_1(\xi)|_{\q_0{\times}\q}$
and $\pi_1(\xi)(\gt q_0,\q_{(t)}^\xi)=\pi_t(\xi)(\gt q_0,\q_{(t)}^\xi)=0$. Therefore  
\begin{equation}\label{ort-v}
\pi_1(\xi)(\gt q_0,\gt v)=0 
\end{equation} 
 and $\gt v\cap\q_0\subset\q_0^{\bar\xi}$. 
By the construction, $\gt v\subset\q_{(\infty)}^\xi$. Thus 
\begin{equation} \label{dimv}
\dim\q_{(\infty)}^\xi\ge \dim(\gt v+\q_0) = \dim\gt v+\dim\q_0-\dim(\q_0\cap \gt v)\ge \ind\q+\dim\q_0-\ind\q_0.
\end{equation} 
This finishes the proof, since the inequality holds on a non-empty open subset. 
\end{proof}

The assumption $\gt q_0^*\cap\gt q^*_{\sf reg}\ne\varnothing$ is satisfied in the reductive case, 
since the reductive subalgebra $\g_0=\g^\vartheta$ contains regular semisimple elements 
of $\g$, see  e.g.~\cite[\S 8.8]{kac}.
Thus, we obtain a new proof of \cite[Theorem 3.2]{fo}. 

\begin{cl}[\cite{fo}]  \label{cor:infty=reg}
Suppose that $\q=\g$ is reductive. 
Then  one has $\infty\in \BP_{\sf reg}$ if and only if\/ $\dim\g_0=\rk\g_0$, i.e., $\g_0$ is an abelian
subalgebra of $\g$. 
\end{cl}

Recall that $\gt q_0\subset\cz_\infty$. Thereby $\cz_\infty$ is not Poisson-commutative, unless $\gt q_0$ is commutative. 
However, if $[\gt q_0,\gt q_0]=0$ and $\gt q_0^*\cap\gt q^*_{\sf reg}\ne\varnothing$, then 
$\ind\q_0=\dim\q_0$ and 
$\{\,\,,\,\}_\infty$ is a regular structure in the pencil  spanned by $\{\,\,,\,\}$ and $\{\,\,,\,\}_0$, see Theorem~\ref{thm:ind-inf}. In this case,
$\cz_\infty\subset\gZ(\q,\vth)$. Thereby a description of $\cz_\infty$ is desirable. 

The group $Q_0$ acts on $\gt q_{(\infty)}$. Hence one may consider $\cz_\infty^{\q_0}\subset \cz_\infty$.  

\begin{thm} \label{inf-inv}
Assume that $\q_0^*\cap\q^*_{\sf reg}\ne\varnothing$. \\[.2ex]
{\sf (i)} We have $\{\cz_\infty^{\q_0},\cz_\infty^{\q_0}\}=0$. \\[.2ex]
{\sf (ii)} The algebra $\mathsf{alg}\lg \gZ,\cz_\infty^{\q_0}\rg$ is still Poisson-commutative. \\[.2ex]
{\sf (iii)} We have also $\trdeg\cz_\infty^{\q_0}\le\ind\q$. 
\end{thm}
\begin{proof}
Take $\xi\in\q^*_{\sf reg}\cap \q^*_{\infty,\sf reg}$ such that  $\bar\xi=\xi|_{\q_0}\in(\q_0^*)_{\sf reg}$.
Note that there is a non-empty open subset consisting of suitable elements. 
Following the proof of Theorem~\ref{thm:ind-inf},
set $\gt v=\lim\limits_{t\to\infty} \q_{(t)}^\xi$. Recall that 
$\gt v+\gt q_0\subset \gt q_{(\infty)}^\xi$ and that $\dim(\gt v+\gt q_0)\ge \ind\q_{(\infty)}$, see~\eqref{dimv}. Since 
$\xi\in\q^*_{\infty,\sf reg}$, there is the equality $\gt q_{(\infty)}^\xi =\gt v+\gt q_0$. 

Next $\textsl{d}_\xi  \cz_\infty^{\q_0} \subset \gt q_{(\infty)}^\xi$ by \eqref{incl} and
$\pi_1(\xi)(\textsl{d}_\xi  \cz_\infty^{\q_0},\q_0)=0$, 
since $\cz_\infty^{\q_0}$ consists of $\q_0$-invariants.  In the proof of Theorem~\ref{thm:ind-inf}, we have established that 
$\pi_1(\xi)(\gt v,\q_0)=0$, see \eqref{ort-v}. Suppose that $y\in\q_0$ and $\pi_1(\xi)(y,\q_0)=0$. Then 
$y\in\q_0^{\bar\xi}$. In particular, $\textsl{d}_\xi  \cz_\infty^{\q_0}\subset \gt v+\q_0^{\bar\xi}$ and $\gt v\cap\q_0\subset\q_0^{\bar\xi}$. By the dimension reasons,
$\gt v\cap\q_0=\q_0^{\bar\xi}$. Thus   $\textsl{d}_\xi  \cz_\infty^{\q_0} \subset\gt v+\q_0^{\bar\xi}\subset\gt v$. 
The inclusion proves the inequality $\trdeg\cz_\infty^{\q_0}\le\dim\gt v=\ind\q$. 

For almost all $t\in\BP$, we have $\dim\q_{(t)}^\xi=\ind\q$. Hence
$\gt v$ is a subspace of 
$$
L(\xi):=\sum\limits_{t:\,\rk\pi_t(\xi)=\dim\q-\ind\q} \q_{(t)}^\xi
$$ and the latter is known to be isotropic  w.r.t. $\pi_t(\xi)$ for  any $t$,
 see e.g. \cite[Appendix]{codim3}. 
Then, in particular, any $F\in\{\cz_\infty^{\q_0},\cz_\infty^{\q_0}\}$ vanishes at $\xi$, and,
since $\xi$ is generic, the first   claim is  settled.  

If $\infty\in \BP_{\sf reg}$, then $\cz_\infty^{\q_0}\subset \cz_\infty\subset\gZ$ and part {\sf (ii)} is clear. 
Suppose that $\infty\in \BP_{\sf sing}$. Then $\gZ\subset\mathsf{alg}\langle\cz_t \mid t\ne\infty\rangle$.
For any $t\ne \infty$, the brackets  $\{\,\,,\,\}_\infty$ and $\{\,\,,\,\}_t$ span 
$$\mathcal P=\{ a\{\,\,,\,\}_0+b\{\,\,,\,\}_\infty\mid (a,b)\in\bbk^2\}
$$
 and
$\{\cz_\infty,\cz_t\}_\infty = \{\cz_\infty,\cz_t\}_t=0$. Thereby $\{\cz_\infty,\cz_t\}=0$ for each $t\ne\infty$. 
This finishes the proof. 
\end{proof}

\section{The reductive case} \label{red}

In most of this section, we recollect known results about automorphisms of reductive Lie algebras and properties of 
$\gZ(\g,\vth)$. 
Statements of Section~\ref{folding} are crucial for the  proof of Theorem~\ref{inf-c2}. 
Theorem~\ref{ind-au} on the index of $\g_{(0)}$ is a new result.

\subsection{The Kac diagram of a finite order automorphism} \label{Kac}
 We describe briefly Kac's classification of 
finite order automorphisms  in the semisimple case \cite{k69}. 

A pair $(\g,\vth)$ is {\it decomposable}, if $\g$ is a direct sum of two non-trivial
$\vth$-stable ideals. Otherwise  $(\g,\vth)$ is said to be {\it indecomposable}. Classification of finite order automorphism 
readily reduces to the indecomposable case. The centre of $\g$ is always a $\vth$-stable ideal and automorphisms of an abelian Lie algebra have no particular significance (in our context). Therefore assume that $\g$ is semisimple. 

Suppose that $\g$ is not simple and $(\g,\vth)$ is  indecomposable. Then $\g=\h^{\oplus n}$ is a sum of $n$ copies of a simple Lie algebra $\h$ and $\vth$ is a composition of an automorphism of $\h$ and a cyclic permutation of the 
summands. 

Below we assume that $\g$ is simple. By a result of R.\,Steinberg~\cite[Theorem\,7.5]{st},
every semisimple automorphism of $\g$ fixes a Borel subalgebra of $\g$ and a Cartan 
subalgebra thereof. Let $\be$ be a $\vth$-stable Borel subalgebra and $\te\subset\be$ a 
$\vth$-stable Cartan subalgebra. This yields a $\vth$-stable triangular decomposition 
$\g=\ut^-\oplus\te\oplus\ut$, where $\ut=[\be,\be]$. Let $\Delta=\Delta(\g)$ be the set of roots of $\g$ related to $\te$, 
$\Delta^+$ the set of positive roots corresponding to $\ut$, and $\Pi\subset\Delta^+$ the set of simple 
roots. Let $\g_\gamma$ be the root space for $\gamma\in\Delta$. Hence $\ut=\bigoplus_{\gamma\in\Delta^+}\g_\gamma$.
Let $e_\gamma\in\g_\gamma$ be a non-zero root vector. 

Clearly, $\vth$ induces a permutation of $\Pi$, which is an automorphism of the Dynkin diagram,
and $\vth$ is inner if and only if this permutation is trivial. Accordingly, $\vth$ can be written as a product 
$\sigma{\circ}\vth'$, where $\vth'$ is inner and $\sigma$ is the so-called {\it diagram automorphism\/} of 
$\g$. We refer to \cite[\S\,8.2]{kac} for an explicit construction and properties of $\sigma$. In particular, 
$\sigma$ depends only on the connected component of ${\sf Aut}(\g)$ that contains $\vth$ and $\Ord$ 
equals the order of the corresponding permutation of $\Pi$. 

--- The case of an inner $\vth$:\\[.2ex]
Set $\Pi=\{\ap_1,\ldots,\ap_l\}$ and let $\delta=\sum_{i=1}^l n_i \ap_i$ be the highest root in $\Delta^+$. 
An inner periodic automorphism with $\te\subset\g_0$ is determined by an $(l+1)$-tuple of non-negative 
integers, {\it Kac labels},  $\boldsymbol{p}=(p_0,p_1,\ldots,p_l)$ such that $\gcd(p_0,\ldots,p_l)=1$ and 
$\boldsymbol{p}\ne (0,\dots,0)$.
Set $m:=p_0+\sum_{i=1}^{l} n_i p_i$ and let $\ov{p_i}$ denote the unique representative of
$\{0,1,\dots,m-1\}$ such that $p_i \equiv \ov{p_i} \pmod m$. The $\BZ_m$-grading 
$\g=\bigoplus_{i=0}^{m-1}\g_i$ corresponding to $\vth=\vth(\boldsymbol{p})$ is defined 
by the conditions that
\[
   \g_{\ap_i}\subset\g_{\ov{p_i}} \ \text{ for $i=1,\dots,l$}, 
   \enskip \g_{-\delta}\subset\g_{\ov{p_0}},  \text{ and } \ \te\subset\g_0.
\]
The {\it Kac diagram\/} $\eus K(\vth)$ of $\vth=\vth(\boldsymbol{p})$ is the {\bf affine} (=\,extended) 
Dynkin diagram of $\g$, $\tilde{\eus D}(\g)$, equipped with the labels $p_0,p_1,\dots,p_l$. In 
$\eus K(\vth)$, the $i$-th node of the usual Dynkin diagram ${\eus D}(\g)$ represents $\ap_i$ and the 
extra node represents $-\delta$. It is convenient to assume that $\ap_0=-\delta$ and $n_0=1$. Then 
$(l+1)$-tuple $(n_0,n_1,\dots,n_l)$ yields the coefficients of a linear dependence for $\ap_0,\ap_1,\dots,\ap_l$.

--- The case of an outer $\vth$:\\[.2ex]
Let $\sigma$ be the diagram automorphism of $\g$ related to $\vth$.
The order of a nontrivial 
diagram automorphism is either $2$ or $3$, there $3$ is possible only for $\gt g$ of type {\sf D}$_4$. 
Therefore, $\sigma$ defines either  a $\BZ_2$- or $\BZ_3$-grading of $\g$, say $\gt g=\gt g^\sigma\oplus\gt g_1^{(\sigma)}$ or 
$\gt g=\gt g^\sigma\oplus\gt g_1^{(\sigma)}\oplus\gt g_2^{(\sigma)}$. 

The Kac diagrams of   outer periodic automorphisms are supported on the twisted affine Dynkin diagrams
of index 2 and 3, see \cite[\S\,8]{vi76} and~\cite[Table\,3]{t41}. Such a diagram has $r+1$ nodes, where 
$r=\rk\g^\sigma$, certain $r$ nodes comprise the Dynkin diagram of the simple Lie algebra $\g^{\sigma}$, 
and the additional node represents the lowest weight $-\delta_1$ of the $\g^\sigma$-module 
$\g_1^{(\sigma)}$. Write $\delta_1=\sum_{i=1}^r a_i' \nu_i$, where the elements $\nu_i$ are the simple roots of 
$\g^\sigma$, and set $a'_0=1$. Then the $(r+1)$-tuple 
$(a'_0,a'_1,\dots,a'_r)$  yields coefficients of linear dependence for $-\delta_1,\nu_1,\dots,\nu_r$.

Let $\boldsymbol{p}=(p_0, p_1,\dots, p_r)$ be an $(r+1)$-tuple of non-negative integers, {\it Kac labels}, such that 
$\boldsymbol{p}\ne (0,0,\dots,0)$ and $\gcd(p_0, p_1,\dots, p_r)=1$. The Kac diagram of 
$\vth=\vth(\boldsymbol{p})$ is the required twisted affine diagram equipped with the labels 
$(p_0, p_1,\dots, p_r)$ over the nodes. Then $m=\mathsf{ord}(\vth(\boldsymbol{p}))=\Ord{\cdot} \sum_{i=0}^r a_i' p_i$. 

\subsection{Relation between roots of $\g$ and $\g^\sigma$.} \label{folding}
Let $\sigma$ be a diagram automorphism of $\gt g$ associated with $\gt b$ and $\te\subset\gt b$.
Then $\sigma(\gt b)=\gt b$ and $\sigma(\te)=\te$ by the construction. Let 
$\Delta^+_{\g^\sigma}$ be the set of positive roots of $\g^\sigma$ associated with $(\gt b^\sigma,\te_0)$, where $\te_0=\te^\sigma$.

Take any $\alpha\in\Delta^+$. If $\sigma(\alpha)\ne\alpha$, then the restriction $\bar\alpha=\alpha|_{\te_0}$ is a positive root of $\g^{\sigma}$. In any case, 
$\bar\alpha=\sigma(\alpha)|_{\te_0}$. 
Suppose $\beta\in \Delta_{\g^\sigma}$. Then there is a non-zero root vector $x_\beta\in\g^\sigma$. We can write 
$x_\beta=\sum_{\alpha\in\Delta} b_\alpha e_\alpha$ with $b_\alpha\in\bbk$. Then 
$\alpha|_{\te_0}=\beta$, whenever $b_\alpha\ne 0$. In particular,  $\Delta_{\g^\sigma}$ is contained in 
$\Delta|_{\te_0}$.

Let $\left<\sigma\right>\subset\Aut(\g)$ be a subgroup generated by $\sigma$. In all types except 
${\sf A}_{2n}$, 
the restriction of roots from $\te$ to $\te_0$ produces 
bijections between $\left<\sigma\right>$-orbits on $\Delta^+$ and 
$\Delta^+_{\g^\sigma}$. If $\g=\gt{sl}_{2n+1}$, then the situation is slightly different. 
These are well-known facts, nevertheless we give a brieft explanation below. Set 
$\gt s_\alpha=\g_\alpha\oplus\g_{-\alpha}\oplus[\g_\alpha,\g_{-\alpha}]$.

\begin{lm}[{cf. \cite[\S\,8.2]{kac}}] \label{bij-D4}
Suppose   $\Ord=3$. Then there is a bijection between $\left<\sigma\right>$-orbits on $\Delta^+$ and 
$\Delta^+_{\g^\sigma}$.
\end{lm} 
\begin{proof} Here $\gt g$ is of type {\sf D}$_4$. 
There are six $\left<\sigma\right>$-orbits on $\Delta^+$. Three of these orbits have three elements,
namely 
$$
\{\alpha_1,\alpha_3,\alpha_4\}, \  \{\alpha_1+\alpha_2,\alpha_2+\alpha_3,\alpha_2+\alpha_4\}, \
 \{\alpha_1+\alpha_2+\alpha_3,\alpha_1+\alpha_2+\alpha_4,\alpha_2+\alpha_3+\alpha_4\}.
 $$
The other three consist of fixed points: $\{\alpha_2\}$, $\{\alpha_1+\alpha_2+\alpha_3+\alpha_4\}$, 
$\{\alpha_1+2\alpha_2+\alpha_3+\alpha_4\}$.
 The Lie algebra $\g^\sigma$ is of type ${\sf G}_2$ and it has 
$6$ positive roots. Therefore there is nothing more to show.   
\end{proof}

\begin{lm} \label{index2-restr}
Suppose $\Ord=2$. 
Let $x_\mu=\sum_{\alpha\in\Delta} b_\alpha e_\alpha\in (\gt b\cap\gt g^\sigma)$ with $b_\alpha\in\bbk$ be a non-zero root vector of $\g^\sigma$.
Fix one $\alpha\in\Delta^+$ such that  $b_\alpha\ne 0$.  
Suppose further  that $\bar\beta=\beta|_{\te_0}\in \bbk\bar\alpha$ for some 
$\beta\in\Delta^+$ that does not belong to the $\left<\sigma\right>$-orbit of $\alpha$.
Then $\beta=\alpha+\sigma(\alpha)$. 
\end{lm}
\begin{proof}
Since $\sigma(x_\mu)=x_\mu$, we have 
$(\gt s_{\alpha}+\sigma(\gt s_\alpha))^\sigma\simeq\gt{sl}_2$, and if 
 $\sigma(\alpha)=\alpha$, then 
$\sigma|_{\g_\alpha}=\id$.  

Let $h\in\te_0$ be a subregular element of $\g^\sigma$ 
such that $\bar\alpha(h)=0$ and $\gamma(h)\ne 0$, whenever $\gamma\in\Delta$ and 
$\gamma|_{\te_0}\not\in\bbk\bar\alpha$. 
Consider $\gt g^h$. On the one hand, we have 
$(\g^h)^{\sigma}=(\g^\sigma)^h =\gt{sl}_2+\te_0$. On the other hand, if $\sigma(\beta)\ne\beta$, then 
\[
(\g^h)^{\sigma}\supset
(\gt s_\alpha+\gt s_{\sigma(\alpha)})^\sigma \oplus (\gt s_\beta+\gt s_{\sigma(\beta)})^\sigma \simeq  \gt{sl}_2\oplus\gt{sl}_2\,,
\]
which is a contradiction. This shows that  $\sigma(\beta)=\beta$.

By the same reason as above,  $\sigma|_{\g_\beta}\ne \id$.
It follows that $\sigma|_{\g_\beta}=\sigma|_{\g_{-\beta}} =-\id$,
since $\Ord=2$. Then $[\g_\beta,\g_{-\beta}] \subset\te_0$. Let
$h_\beta\in[\g_\beta,\g_{-\beta}]$ be such that  $\beta(h_\beta)=2$.
Since $h_\beta\in \te_0$ and $\bar\beta\in\bbk\bar\alpha$, we have $\alpha(h_\beta)\ne 0$.  
In view of this, $[\gt s_\beta,\gt s_\alpha]\ne 0$. 
 
Set $\gt f=[\g^h,\g^h]$. 
Then $\gt s_\alpha+\gt s_\beta\subset\gt f$. Furthermore, if $\gamma\in\Delta$ is a root  of $\gt f$,
then $\bar\gamma=\gamma|_{\te_0}$ belongs to 
$\bbk\bar\alpha$. Hence, 
$\rk\gt f\ge 2$ and $\gt f$ is simple. 
By the construction, 
$\gt f^\sigma=\gt{sl}_2+\tilde\te$, where $\tilde\te\subset\te_0$ is a maximal torus of  
$\gt f^\sigma$. 
The involution $\sigma$ induces an automorphism of $\gt f$ of order $2$. It cannot be inner, because 
the restrictions of $\alpha$ and $\beta$ to $\tilde\te$ coincide. From the description  of outer involutions, we deduce that 
$\rk\gt f\le 2$. Therefore 
the only possibility for $(\gt f,\gt f^\sigma)$ is the pair $(\gt{sl}_3,\gt{so}_3)$. 
Here   $\beta=\alpha_1+\alpha_2$ with $\alpha_1,\alpha_2$ being  simple roots of $\gt{sl}_3$ and 
$\sigma(\alpha_1)=\alpha_2$.  
Since $e_\beta=[e_{\alpha_1},e_{\alpha_2}]$ in $\gt{sl}_3$, up to a suitable normalisation, we have also 
$e_\beta=[e_{\alpha},e_{\sigma(\alpha)}]$ in $\gt g$ and $\beta=\alpha+\sigma(\alpha)$ in $\Delta$.
\end{proof}


\subsection{Properties and generators of algebras $\gZ(\g,\vth)$}
\label{sect:3}

\noindent
From now on, $G$ is a connected semisimple algebraic group and $\g=\Lie G$.
We consider $\vth\in\Aut(\g)$ of order $m\ge 2$ and 
freely use the previous notation and results, with $\q$ being replaced by $\g$. In particular,
\[
  \g=\g_0\oplus\g_1\oplus\ldots\oplus \g_{m-1} ,
\] 
where $\{0,1,{\dots},m-1\}$ is the fixed set of representatives for $\BZ_m$, and $G_0$ is the connected subgroup of $G$ with $\Lie G_0=\g_0$. Then $\g_{(t)}$ is a family of Lie algebras parameterised by 
$t\in \BP=\bbk\cup\{\infty\}$, where the algebras $\g_{(t)}$ with $t\in\bbk^\star$ 
are isomorphic to $\g=\g_{(1)}$, while $\g_{(0)}$ and $\g_{(\infty)}$ are different $\BN_0$-graded 
contractions of $\g$. 

 Note that $\g_0$ is
a reductive Lie algebra. Let $\kappa$ be the Killing form on $\g$. We identify $\g$ and $\g_0$ with 
their duals via $\kappa$. Moreover, since $\kappa(\g_i,\g_{j})=0$ if $i+j\not\in\{0,m\}$, the dual space of 
$\g_j$, $\g_j^*$, can be identified with $\g_{m-j}$. We identify also $\te$ with $\te^*$ and $\te_0$ with $\te_0^*$. 
Set $\gZ=\gZ(\g,\vth)$.

\begin{thm}[\cite{fo}]
Suppose that $\ind\g_{(0)}=\rk\g$. Then 
$\trdeg\gZ=\bb(\g,\vth):=\bb(\g)-\bb(\g_0)+\rk\g_0$. \qed
\end{thm}

Note that $\gZ\subset\gS(\g)^{\g_0}$ by \cite[Eq.\,(3${\cdot}$6)]{fo}. Thereby $\trdeg\gZ\le\bb(\g,\vth)$, see~\cite[Prop.\,1.1]{m-y}. 
Thus, if  $\ind\g_{(0)}=\rk\g$, then $\trdeg\gZ$ takes the maximal possible value. 
Note also that $\bb(\g,\vth)=\bb(\g)$ if and only if $[\g_0,\g_0]=0$. 

Since $\vth$ acts on $\gS(\g)^{\g}$, there is a generating
set  $\{H_1,\dots,H_l\}\subset\gS(\g)^{\g}$ consisting of $\vartheta$-eigenvectors.
Then  $\vartheta(H_i)=\zeta^{r_i} H_i$ with $0\le r_i<m$. The integers $r_i$ depend only on the
connected component of $\Aut(\g)$ that contains $\vth$, and if $a$ is the 
order of $\vth$ in ${\Aut}(\g)/\mathsf{Int}(\g)$, then $\zeta^{ar_i}=1$. Therefore, if $\g$ is simple, then 
$\zeta^{r_i} = \pm 1$ for all types but $\GR{D}{4}$.

Recall from Section~\ref{subs:periodic} that $\g_{(0)}=\g_{(0,\varphi)}=:\g_{(0,\vth)}$ is a contraction of $\g$ defined by $\varphi$. 
Below we freely use notation of Sections~\ref{subs:contr-&-inv},~\ref{sec-ggs}. 

\begin{lm}[\cite{fo}]      \label{lm:outer-inv}
For any  $\vartheta\in {\sf Aut}(\g)$ of order $m$, we have

{\sf (1)} \ $\vartheta(H_j)=H_j$ if and only if $d^\bullet_j\in m\Z$;

{\sf (2)} \ $\sum_{j=1}^l r_j = \frac{1}{2}m(\rk\g-\rk\g_0)$;

{\sf (3)} \ $\rk\g_0=\# \{j \mid \vartheta(H_j)=H_j\}$. \qed
\end{lm}

\begin{lm} \label{lm-restr}
For each $j$,
the restriction of $H_j$ to $\g_0^*$ is non-zero if and only if $\vth(H_j)=H_j$.
\end{lm}
\begin{proof}
If ${H_j}|_{\g_0^*}\ne 0$, then the lowest $\varphi$-component  $H_{j,0}\in\gS(\g_0)$ is non-zero and hence we have 
$\vth(H_j)=H_j$. 

Consider some $x\in\te_0^*\cap \g_{\sf reg}^*$, which exists by~\cite[\S 8.8]{kac},
and apply Kostant's regularity criterion~\cite[Theorem~9]{ko63} to $x$. 
According to this criterion,
$\lg \textsl{d}_x H_i\mid 1\le i\le l\rg_{\bbk}=\g^x=\te$. 
Here $\textsl{d}_x H_i\in \g_0^x=\te_0$ if and only if $\vth(H_i)=H_i$.
In view of Lemma~\ref{lm:outer-inv}{\sf (3)}, we have $\textsl{d}_x H_i \ne 0$ for each $i$ such that $r_i=0$. Then also 
$H_{i,0}\ne 0$, whenever $r_i=0$.
\end{proof}

\begin{thm}    [\cite{fo}]                       \label{free-main}
Suppose that $\vth\in\Aut(\g)$ admits a {\sf g.g.s.} and\/ $\ind\g_{(0)}=\rk\g$. Then 
\\ \indent
{\sf (i)} $\gZ_{\times}:=\mathsf{alg}\langle H_{j,i} \mid 1\le j\le l, 0\le i \le d_j^\bullet \rangle\subset\gZ$ is a polynomial Poisson-commutative sub\-alge\-bra of\/ $\gS(\g)^{\g_0}$ having 
the maximal transcendence degree. 
\\ \indent
{\sf (ii)} More precisely,
if $H_1,\dots,H_l$ is a {\sf g.g.s.} that consists of $\vth$-eigenvectors, then
$\gZ_{\times}$ is freely generated by the non-zero bi-homogeneous components of all $H_j$. \qed
\end{thm}

A precise relationship between $\gZ$ and $\gZ_\times$ depends on further properties of $\vth$. Two
complementary assertion are given below.

\begin{cl}[\cite{fo}]      \label{cor1:main4}
In addition to the hypotheses of \emph{Theorem~\ref{free-main}},
suppose that  $\g_{(0)}$ has the {\sl codim}--$2$ property and $\g_0=\g^\vth$ is \emph{not} abelian.
Then $\gZ=\gZ_\times$ is the polynomial algebra freely generated by all non-zero bi-homogeneous components $H_{j,i}$.
\end{cl}

\begin{cl}  [\cite{fo}]     \label{cor2:main4}
In addition to the hypotheses of \emph{Theorem~\ref{free-main}},
suppose that $\g_{(0)}$ has the {\sl codim}--$2$ property, $\vth$ is inner, and $\g_0=\g^\vth$ \emph{is} abelian. Then $\cz_\infty=\gS(\g_0)$ and $\gZ=\mathsf{alg}\lg \gZ_{\times},\g_0\rg$ is a polynomial algebra.
 \end{cl}

\subsection{The equality for the index of $\g_{(0)}$} \label{sec-new}
The equality $\ind\g_{(0)}=\rk\g$ is very important in our context. If it holds, then 
$\gZ$ can be extended to a Poisson-commutative subalgebra of the maximal possible transcendence degree $\bb(\g)$ 
in the same way as in \cite[Sect.\,6.2]{OY}.

\textbullet \ It holds if $\oth$ is two or three \cite{p07,MZ23}. 

\textbullet \ It holds if $\g$ is either $\gt{so}_n$ or of type ${\sf G}_2$   \cite{MZ23}. 

\textbullet \  It holds if $\g_1\cap\g_{\sf reg}\ne\varnothing$ 
by  \cite[Prop.\,5.3]{p09}. 

\noindent
Below we will see a few more positive examples.

Let $(\g,\vth)$ be an indecomposable pair, where $\g=\h^{\oplus n}$, the algebra 
$\h$ is simple, and $\vth$ corresponds to $\tilde\vth\in\Aut(\h)$. Note that $\tilde\vth$ may be trivial.  
We remark also that $\gt g_0\simeq\gt h^{\tilde\vth}$. 

Let $\gt h_{(0)}$ be the contraction of $\gt h$ associated with $\tilde\vth$. 

\begin{lm}[{\cite[Lemma~8.1]{fo}}]\label{ind-d}
 If $\ind\h_{(0)}=\rk\h$, then 
$\ind\g_{(0)}=\rk\g$. If there is a   {\sf g.g.s.} for $\tilde\vth$ in $\gS(\h)$, 
then there is a  {\sf g.g.s.} for $\vth$ in $\gS(\g)$. 
\qed \end{lm}

\begin{thm} \label{ind-au}
Suppose that either $\g_0=\te_0\subset\te$ or $\g$ is simple and among the Kac labels of $\vth$, see Section~\ref{Kac} for the definition, only $p_0$ is zero. 
Then $\ind\g_{(0)}=\rk\g$.
\end{thm}
\begin{proof}
Making use of Lemma~\ref{ind-d}, we may assume that $\g$ is simple. 
If $\gt g_0=\te_0$, then each Kac label of $\vth$ is non-zero, see \cite[Prop.\,17]{vi76} or e.g. \cite[Sect.\,2.4]{MZ23}. 
Let $\vth_1$ be the automorphism obtained from $\vth$ by changing each non-zero Kac label of  
$\vth$ to $1$. Then $\vth_1$ may be different from $\vth$ and it may define a different
$\BZ_m$-grading of $\g$. However, the Lie brackets $[\,\,,\,]_{(0)}$ defined by $\vth$ and $\vth_1$ coincide \cite{MZ23}.
Thus we may suppose that each non-zero Kac label of $\vth$ is $1$. By our assumptions on $\vth$, only $p_0$ may not be equal to $1$.

Let $\Pi$ be a set of the simple roots of $\g$ chosen in the same way as in Section~\ref{Kac}. 
Let $e_i\in\g_{\alpha_i}$ be a non-zero root vector corresponding to a simple root $\alpha_i\in\Pi$. 
If $\vth$ is inner, then $\e=\sum_{i=1}^l e_i \in\g_1$ is a regular nilpotent element in $\g$. 
Thereby $\ind\g_{(0)}=\rk\g$ by  \cite[Prop.\,5.3]{p09}. 

Suppose that $\vth$ is outer. 
Let $\sigma$ be the diagram automorphism of $\g$ associated with $\vth$ and
$\Pi'=\{\nu_1,\ldots,\nu_r\}$  the set of simple roots of $\g^\sigma$ used in 
Section~\ref{Kac}. Then there is a bijection between $\Pi'$ and the set of $\left<\sigma\right>$-orbits on $\Pi$. 
Namely,  $\nu_j$ corresponds to the   $\left<\sigma\right>$-orbit $ \left<\sigma\right>\!{\cdot}\/\alpha_i$ of $\alpha_i\in \Pi$,  if 
$\nu_j=\alpha_i|_{\te_0}$. There is a normalisation of the root vectors $e_i=e_{\alpha_i}$ such that 
$e'_j=\sum_{\alpha\in \left<\sigma\right>{\cdot}\alpha_i} e_\alpha$ is a simple 
root vector of $\g^{\sigma}$ of the weight $\nu_j={\alpha_i}|_{\te_0}$ for each $j$ \cite[\S\,8.2]{kac}. 
If the Kac label $p_j$ is $1$, then $e'_j\in \g_1$ according to the description of $\g_1$ given in
\cite[Prop.\,17]{vi76}. Since $p_j=1$ for any $j\ge 1$, we obtain $\e=\sum_{i=1}^l e_i \in\g_1$. 
Then again  by \cite[Prop.\,5.3]{p09}, we have $\ind\g_{(0)}=\rk\g$. 
\end{proof}

The case $\g_0=\te_0$ is of particular importance, because $\trdeg\gZ(\g,\vth)=\bb(\g)$ here. 
If $\gt g_0=\te$, then 
$\gZ(\g,\vth)$ coincides with the Poisson-commutative subalgebra $\gZ(\be,\ut^-)$ constructed in \cite{bn-oy}, 
see \cite[Example\,4.5]{MZ23}. In particular, this algebra is known to be maximal~\cite[Theorem\,5.5]{bn-oy}. 
If $\gt g_0=\te_0$ is a proper subspace of $\te$, then we obtain a less studied subalgebra. 
We conjecture that it
 is still maximal. 

\section{Properties of  $\cz_\infty$ in the reductive case}
\label{sect:infty-g}

In order to understand $\cz_\infty^{\g_0}$, we study symmetric invariants of
$\tilde\g =\g_0\ltimes\g_{(\infty)}$. 
The Lie algebra $\tilde\g$ is a contraction of $\g_0\oplus\g$ associated with a map
$\tilde\vp\!:\bbk^\star\to\GL(\g_0\oplus\g)$, where $\tilde\vp(s)=\tilde\vp_s$. In order to define 
$\tilde\vp_s$, we consider a vector space decomposition 
\[
\g_0\oplus\g=\g_0^{\rm d} \oplus\g_{m-1}\oplus\ldots\oplus\g_2\oplus\g_1\oplus\g_0^{\rm ab},
\]
where the first summand $\g_0^{\rm d}$ is embedded diagonally into $\g_0\oplus\g$ and the last summand 
$\g_0^{\rm ab}$ is embedded anti-diagonally. Then set ${\tilde\vp_s}|_{\g_0^{\rm d}}=\id$, 
${\tilde\vp_s}|_{\g_{m-j}}=s^j\id$, and finally ${\tilde\vp_s}|_{\g_{0}^{\rm ab}}=s^m\id$. If we consider 
$\xi\in(\g_0^{\rm ab})^*$ as an element of $\g^*$ and of $\tilde\g^*$, then 
$\tilde\g^{\xi}=\g_0^{\xi}\oplus\g^{\xi}$ as a vector space. If $\xi$ is regular in $\g^*$, then 
$\dim\tilde\g^\xi=\ind\g_0+\ind\g$, cf.~\eqref{q0-inf}. 
Then, in view of ~\cite[\S 8.8]{kac}, which states that $\gt g_0^*\cap\g^*_{\sf reg}\ne\varnothing$, we have
$\ind\tilde\g\le\ind\gt g_0+\ind\g$. Since $\tilde\g$ is a contraction of $\g_0\oplus\g$, there is the 
equality, cf. \cite[Sect.\,6.2]{p09},  
\begin{equation}\label{eq4}
\ind\tilde\g=\ind\!(\gt g_0\oplus\g)=\rk\g_0+\rk\g. 
\end{equation}

Note that $\vth$ is also an automorphism of $\g_0\oplus\g$, acting on $\g_0\oplus\g_0$ trivially,
as well as of $\tilde\g$. However, the contraction defined by $\tilde\vp$ is not directly related to $\vth$. 

\subsection{Small rank examples} \label{smallr}

Let $\vth$ be the outer automorphism of $\gt{sl}_3$ of order $4$ with the Kac labels 
$\boldsymbol{p}=(0,1)$.
Then the corresponding $\Z_4$-grading of $\gt g=\gt{sl}_3$ looks as follows 
$$
\gt g=\gt{sl}_2\oplus\bbk^2_{I}\oplus\te_1\oplus\bbk^2_{II},
$$
where $\dim\bbk^2_{I}=\dim \bbk^2_{II}=2$ and $\te_1\subset\te$ is one-dimensional. 

Let $E_{ij}\in\gt{gl}_n$ be elementary matrices (matrix units). 

\begin{lm} \label{sl3}
Suppose $\tilde\g=\g_0\ltimes\g_{(\infty)}$ corresponds  to the pair $(\g,\vth)$ described above. 
Let $e\in(\gt g_0^{\rm ab})^*$ be a non-zero nilpotent element. 
Then $e+y\in \tilde\g^*_{\sf reg}$\, for  a generic $y\in\g_1^*$. 
\end{lm}
\begin{proof}
We may suppose that $\gt g_0=\gt{sl}_2$ is embedded into $\gt g$ as 
$\left<E_{11}-E_{33},E_{13},E_{31}\right>_{\bbk}$. Then 
$\bbk^2_{I}=\left<E_{12}-E_{23},E_{21}+E_{32}\right>_{\bbk}$, 
$\bbk^2_{II}=\left<E_{12}+E_{23},E_{21}-E_{32}\right>_{\bbk}$,
$\te=\bbk(E_{11}-2E_{22}+E_{33})$, and $e$ is identified with 
$\kappa(E_{13},\,)$. As $y$ we choose $\kappa(E_{21}-E_{32},\,)$. Set $\gamma=y+e$ and let 
$\pi$ be the Poisson tensor of $\tilde\g$. We identify $\gt g_0^{\rm d}$ and $\gt g_0^{\rm ab}$
with $\gt{sl}_2$ in a natural way. 

Observe that $\pi(\gamma)(\te_1,\bbk(E_{12}+E_{23}))\ne 0$ and that 
\begin{equation} \label{sl3-s3}
\pi(\gamma)(\te_1+\bbk(E_{12}+E_{23}),\g_0^{\rm d}+\bbk^2_{I}+\bbk(E_{21}-E_{32})+ \g_0^{\rm ab})=0.
\end{equation}
Therefore $\tilde\g^\gamma\subset  \g_0^{\rm d}+\bbk^2_{I}+\bbk(E_{21}-E_{32})+ \g_0^{\rm ab}$. 
Clearly $E_{13}\in\g_0^{\rm ab}$ belongs to $\tilde\g^\gamma$. Now suppose that 
$x=x_0+x_1+x_3+x_4$ with $x_0\in\g_0^{\rm d}$, $x_1\in\bbk(E_{21}-E_{32})$, 
$x_3\in\bbk^2_{I}$, $x_4\in\g_0^{\rm ab}$ belongs to $\tilde\g^\gamma$. 
We may assume that $x_4\in\bbk E_{31}+\bbk(E_{11}-E_{33})$. 
Then 
\begin{align} \label{sl3-s}
\pi(\gamma)(x,\bbk^2_{I})=\pi(\gamma)(x_0+x_1,\bbk^2_{I})=0 , \ 
\pi(\gamma)(x,\g_0^{\rm ab})=\pi(\gamma)(x_0,\g_0^{\rm ab})=0, 
\enskip \text{ and } \\
\pi(\gamma)(x,\g_0^{\rm d})=\pi(\gamma)(x_3+x_4,\g_0^{\rm d})=0. \label{sl3-s2}
\end{align}
Thereby $x_0\in\bbk E_{13}$, which implies $x_0+x_1\in\bbk(E_{13}+E_{21}-E_{32})$. 
Since $\pi(\gamma)(x_4,E_{13})=0$ and $\pi(\gamma)(E_{21}+E_{32},E_{13})\ne 0$ for $E_{13}\in\g_0^{\rm d}$, we obtain  
$x_3\in\bbk(E_{12}-E_{23})$. Looking at the commutators with $E_{11}-E_{33}\in\g_0^{\rm d}$, we obtain $x_3+x_4\in\bbk(E_{12}-E_{23}+E_{31})$. 

We colnclude that $\dim\tilde\g^\gamma\le 3$. It cannot be smaller than $3=\ind\tilde\g$, see~\eqref{eq4}. Thus 
$\gamma$ is a regular point and the same holds for the elements  of a non-empty open subset of 
$\g_1^*$. 
\end{proof}

Next we  generalise Lemma~\ref{sl3} to periodic automorphisms of order $4n$ of direct sums  $\gt{sl}_3^{\oplus n}$. 
Let now $\tilde\vth$ be the outer automorphism of $\gt h=\gt{sl}_3$ of order $4$ with the Kac labels $(1,0)$. 
Set $\gt g=\gt h^{\oplus n}$ and let $\vth\in\Aut(\g)$ be obtained from $\tilde\vth$ and a cyclic permutation of the summands 
of $\g$, see~\eqref{tildet}.

\begin{lm} \label{sl3-n}
Suppose $\tilde\g=\g_0\ltimes\g_{(\infty)}$ corresponds  to the pair $(\g,\vth)$ described above. 
Let $e\in(\gt g_0^{\rm ab})^*$ be a non-zero nilpotent element. 
Then $e+y\in \tilde\g^*_{\sf reg}$\, for  a generic $y\in\g_1^*$.  
\end{lm}
\begin{proof}
It is convenient to identify $\g_0^{\rm d}$ with $\gt h_0$,   $\gt g_0^{\rm ab}$ with $\gt h_0 \bar t^{4n}$, 
each $\gt g_{4k+i}$, where $0\le i\le 3$, with 
$\gt h_i \bar t^{4n-4k-i}$, assuming that $\bar t^{4n+1}=0$. Then 
$[x\bar t^j, y\bar t^u]=[x,y]\bar t^{j+u}$ in $\tilde\g$ for $\tilde\vth$-eigenvectors $x,y\in\gt h$.

We understand $e$ as a linear function such that $e(E_{13}\bar t^{4n})=e((E_{11}-E_{33})\bar t^{4n})=0$ and 
$e(E_{31}\bar t^{4n})=1$. Modelling the proof of Lemma~\ref{sl3},
take $y$ such that $y((E_{12}-E_{23})\bar t^{4n-1})=2$ and $y((E_{21}+E_{32})\bar t^{4n-1})= 0$. 
Set $\gamma=e+y$. 
Note that $\gamma([x\bar t^j, y\bar t^u])=0$, if $j+u\not\in\{4n-1,4n\}$. Iterating computations in the spirit of \eqref{sl3-s3},~\eqref{sl3-s},~\eqref{sl3-s2}, we obtain 
$$
\tilde\g^{\gamma}=\left<E_{13}\bar t^{4n}, \ E_{13}\bar t^{4k}+(E_{21}-E_{32})\bar t^{4k+1}, \
(E_{12}-E_{23})\bar t^{4k+3}+E_{31}\bar t^{4k+4} 
 \mid 0\le k < n    \right>_{\bbk}.
$$
In particular, $\dim\tilde\g^\gamma=1+2n=\ind\tilde\g$ and $\gamma\in\tilde\g^*_{\sf reg}$.
The same holds for the elements  of a non-empty open subset of 
$\g_1^*$. 
\end{proof}

\subsection{General computations}
Let us begin with a  technical lemma. 

\begin{lm}\label{sum}
Let $\gt q$ be a finite-dimensional Lie algebra. 
Suppose we have $\eta\in\gt q^*$ and $\bar y\in(\gt q^\eta)^*$ such that 
$\dim(\gt q^\eta)^{\bar y}=\ind\gt q$. Then for any extension $y$ 
of $\bar y$ to a linear function on $\gt q$, there is $s\in\bbk^\star$ such that 
$s\eta+y\in\gt q^*_{\sf reg}$.
\end{lm}
\begin{proof}
Write $\gt q=\gt m\oplus\gt q^\eta$, where $\hat\eta$ is non-degenerate on $\gt m$. 
Let $\gt m_1\subset\gt q^\eta$ be a subspace of dimension $\dim\gt q^\eta-\ind\gt q$ such that 
$\hat{\bar y}$ is non-degenerate on $\gt m_1$. By the construction  $\dim(\gt m\oplus\gt m_1)=\dim\gt q-\ind\gt q$. 
We extend $\bar y$ to a linear function $y$ on $\gt q$ in some way, which is of no importance. 
Choose 
bases in $\gt m$, $\gt m_1$, take their union,  and let 
$M=\begin{pmatrix} sA+C & B \\ -B^t & A_1\end{pmatrix}$ be the matrix of $s\hat\eta+\hat y$, where $s\in\bbk$, 
in this basis; here $A$ is the matrix of $\hat\eta|_{\gt m}$ and $A_1$ is the matrix of $\hat{\bar y}|_{\gt m_1}$.
Then  
$$
\det(M)=s^{\dim m}\det(A)\det(A_1)+(\,\text{terms of smaller degree in $s$}\,). 
$$
Since $\det(A)\det(A_1)\ne 0$, for almost all $s\in\bbk$, we have $\det(M)\ne 0$. Whenever this happens, 
$\rk(s\hat\eta+\hat y)\ge \dim\gt q-\ind\gt q$ and $s\eta+y\in\gt q^*_{\sf reg}$. 
\end{proof}

\begin{thm} \label{inf-c2}
The algebra $\tilde\g$ has the {\sl codim}--$2$ property.
\end{thm}
\begin{proof}
The subalgebra $\g_0$ is reductive and it contains a semisimple element $x$ that is regular in $\g$, see  e.g.~\cite[\S 8.8]{kac}.
Thus, there is $\eta\in(\g_0^{\rm ab})^*\subset\tilde\g^*$ that corresponds to a regular element of $\g$ and here 
$\dim\tilde\g^\eta=\rk\g+\rk\g_0=\ind\tilde\g$, cf.~\eqref{q0-inf}.  

Take now $\xi\in\tilde\g^*$ such that 
$\bar\xi=\xi|_{\g_0^{\rm ab}}$ corresponds to a regular element of $\g$.   
Note that $\lim\limits_{s\to 0} s^m\tilde\vp_s(\xi)=\bar\xi$. Since $\bar\xi\in\tilde\g^*_{\sf reg}$,  we have 
$\vp_s(\xi)\in\tilde\g^*_{\sf reg}$ for almost all $s$. 
Each map $\tilde\vp_s\!:\tilde\g\to\tilde\g$ is an 
automorphism of the Lie algebra $\tilde\g$.  Thus $\xi\in\tilde\g^*_{\sf reg}$. 

Assume that $D\subset\tilde\g^*_{\sf sing}$ is a divisor in $\tilde\g^*$. Then 
$\{\bar \xi\mid \xi\in D\}\subset(\g_0^{\rm ab})^*$ lies in a proper closed subset of $(\gt g_0^{\rm ab})^*$. 
Thereby $D=D_0\times\Ann(\gt g_0^{\rm ab})$ for some divisor $D_0\subset(\gt g_0^{\rm ab})^*$.
Since $\tilde\g^*_{\sf sing}$ is a $G_0$-stable subset of $\tilde\g^*$, the divisor $D_0$ is $G_0$-stable as well. 
Since $\tilde\g^*_{\sf sing}$ is a conical subset, $D_0$ is the zero set of a homogeneous polynomial. 

A generic fibre of the categorical  quotient $\varpi\!:(\gt g_0^{\rm ab})^*\to (\gt g_0^{\rm ab})^*/\!\!/G_0$ consists of a single
$G_0$-orbit. Since $D_0$ is $G_0$-stable, it follows that   $\dim\overline{\varpi(D_0)}=\rk\g_0-1$. 
Then also  $\varpi^{-1}(y)\subset D_0$ for any $y\in\varpi(D_0)$, because each fibre of $\varpi$ is irreducible.  

The intersection $D_0\cap\gt t_0^*$ does not contain regular in $\g$ elements, thereby  it is 
a subspace of codimension $1$, the zero set of $\bar\alpha=\alpha|_{\te_0}$ for  some  $\alpha\in\Delta^+$.
Let $\eta\in D_0\cap\gt t_0^*$ be a generic point. Then $\eta$ is either regular or subregular in $\g_0$.
Furthermore, 
if  $\beta\in\Delta$ and $\beta(\eta)=0$, then $\bar\beta=\beta|_{\te_0}\in\bbk\bar\alpha$. 
Set $\gt f=[\g^{\eta},\g^{\eta}]$. Note that $\gt f$ is semisimple and that 
 the pair $(\gt f,\gt f^\vth)$  is indecomposable. 

($\diamond$) Case 1. 
Suppose that $\eta$ is not regular in $\g_0$.  Then necessarily $[\g_0^{\eta},\g_0^{\eta}]=\gt{sl}_2$. 
Let $e\in[\g_0^{\eta},\g_0^{\eta}]$ be a non-zero nilpotent element, which we regard also as a point in $(\g_0^{\rm ab})^*$. Then
$\eta+e\in D_0$, since 
$\varpi(\eta+e)=\varpi(\eta)$. 
Suppose $\{\beta\in\Delta^+ \mid \beta(\eta)=0\}=\{\vth^k(\alpha) \mid k\ge 0\}$.
Then
$\vth^k(\alpha)+\vth^{k'}(\alpha)$ with $k,k'\ge 0$ is never a root of $\g$, the subalgebra 
$\gt f$ is a direct sum of copies of $\gt{sl}_2$, the nilpotent element $e$ is regular in $\gt f$, and 
$\eta+e\in\g^*_{\sf reg}$.
This is a contradiction. 

Suppose that there is no equality for the sets above. Then $\vth|_{\gt f}$ is constructed from  an outer automorphism $\tilde\vth$ of a simple 
Lie algebra $\gt h$.  From Section~\ref{folding} we know that $\gt h=\gt{sl}_3$. 
Next $\gt h^{\tilde\vth}\simeq\gt{sl}_2$. The fixed points subalgebra can be embedded in two different ways, as $\gt{so}_3$ or as $\left<E_{11}-E_{33},E_{13},E_{31}\right>_{\bbk}$. For the first embedding, $e$ is still regular in $\gt f$, a contradiction. 

Suppose that $\gt h^{\tilde\vth}=\left<E_{11}-E_{33},E_{13},E_{31}\right>_{\bbk}$.
The twisted affine Dynkin digram of type ${\sf A}_2$ has two nodes. Our choice of $\gt h^{\tilde\vth}$ implies that 
$\tilde\vth$ is the automorphism considered in Lemma~\ref{sl3}.
Then by Lemma~\ref{sl3-n}, there is $y\in\gt f_1^*\simeq \gt f_{i}\subset\gt g_j$, where $0<j<m$  and $i$ depends on $\vth|_{\gt f}$, such that 
$e+y\in\tilde{\gt f}^*_{\sf reg}$.  Note that $s\eta+e+y\in D$ for any $s\in\bbk^\star$. 

The vector spaces $\tilde\g^\eta$ and $(\g_0\oplus\g)^\eta$ coinside.  
We have $\tilde\g^\eta=\tilde{\gt f}+\te_0+\te$, where 
$\tilde\g^\eta$ and $\tilde{\gt f}$ are contractions of $\g^\eta$ and $\gt f\oplus\gt f^{\vth}$, respectively, given by 
the restrictions of $\tilde\vp$. 
We extend $e$ and $y$ to linear functions on $\tilde\g^\eta$ and $\tilde\g$ keeping the same symbols for the extensions. 
Then 
\begin{equation}\label{tilde}
\dim(\tilde\g^\eta)^{e+y}=(\rk\g+\rk\g_0)-(\rk\gt f+\rk\gt f_0)+\ind\tilde{\gt f}=\ind\tilde\g.
\end{equation}
By Lemma~\ref{sum}, $s\eta+e+y\in\tilde\g^*_{\sf reg}$ for some $s\in\bbk^\star$, 
a contradiction. 

($\diamond$) Case 2.  Suppose now that $\eta$ is regular in $\g_0$.  
Then $\gt f_0\subset\te_0$.  
The quotient of $\tilde{\gt f}$ by $(\te_0^{\rm ab}\cap\tilde{\gt f})\subset\g_0^{\rm ab}$ is the contraction 
$\gt f_{(0),\vth^{-1}}$ associated with the restriction of $\vth^{-1}$ to $\gt f$. 
For $y\in\tilde{\gt f}^*$ such that $y(\te_0^{\rm ab})=0$, we have 
$\tilde{\gt f}^y=\gt f_{(0,\vth^{-1})}^y\oplus(\te_0^{\rm ab}\cap\tilde{\gt f})$. 
By Theorem~\ref{ind-au}, $\ind \gt f_{(0),\vth^{-1}}=\rk\gt f$. Therefore there is $y\in\tilde\g^*$ such that 
$y(\te_0^{\rm ab})=0$ and $\dim\tilde{\gt f}^{\bar y}=\rk\gt f+\rk\gt f_0$ for $\bar y=y|_{\tilde{\gt f}}$. 
Repeating the argument of~\eqref{tilde} and using again   Lemma~\ref{sum}, we conclude  
that $s\eta+y\in\tilde\g_{\sf reg}$ for some $s\in\bbk^\star$. This final contradiction shows that there is no 
divisor $D$. 
\end{proof}

\subsection{Symmetric invariants} \label{inf-sym}

Set $r=\rk\g_0$ and choose a set of homogeneous generators $F_1,\ldots,F_r\in\gS(\g_0)^{\g_0}$.
As above, let $\g_0^{\rm ab}$ stand for the abelian ideal of $\tilde\g$, which is isomorphic to $\g_0$ as a $\g_0$-module. 
Let further $\{H_j\mid 1\le j\le l\}\subset\gS(\g)^{\g}$ be a generating set consisting of homogeneous polynomials that are
 $\vth$-eigenvectors. 
 Assume that $\vth(H_j)=H_j$ if and only if $j\le r$, cf. Lemma~\ref{lm:outer-inv}. 
 Unless stated otherwise,  
  we use upper bullets for the highest $\tilde\vp$-components of $H\in\gS(\g_0\oplus\g)$. 

\begin{thm} \label{ggs-inf}
There is  a {\sf g.g.s.}  $F_i, \tilde H_i, H_j$ with $1\le i\le r<j\le l$ for the contraction 
$\gt g_0\oplus\gt g\ard\tilde\g$ defined by $\tilde\varphi$. Furthermore,  for $i\le r$,
$\tilde H_i\in H_i+ \gS(\g_0)^{\g_0}$, where $\g_0$ is a direct summand of $\g_0\oplus\g$,   and 
the ring  
$\gS(\tilde\g)^{\tilde\g}$ is freely generated by $\{F_i^\bullet, \tilde H_i^\bullet, H_j^\bullet\mid 1\le i\le r<j\le l\}$.
\end{thm}
\begin{proof} 
 First  we compute the number $D_{\tilde\vp}$ using the properties of the grading on $\tilde\g$:
\[
D_{\tilde\vp}=m\dim\g_0+\sum_{j=1}^{m-1}(m{-}j)\dim\g_j=m\dim\g_0+\frac{m}{2}(\dim\g-\dim\g_0)=\frac{m}{2}
(\dim\g+\dim\g_0). 
\]
Since $\g_0^{\rm ab}$ is embedded anti-diagonally in $\g_0\oplus\gt g$, we have
$\deg_{\tilde\varphi} F_i^\bullet = m\deg F_i$. 
 
Suppose that $\vth(H_j)=H_j$. Then 
the restriction of $H_j$ from $\g^*$ to $\g_0^*\subset\g^*$ is non-zero by Lemma~\ref{lm-restr}. Thus 
$H_j^\bullet\in\gS(\g_0^{\rm ab})^{\g_0}$. Hence there is a homogeneous polynomial 
 $\tilde H_j$ in $H_j+\gS(\g_0)^{\g_0}$, where $\g_0$ is a direct summand of $\g_0\oplus\g$, such that 
$\tilde d_j^\bullet=\deg_{\tilde\varphi} \tilde H_j^\bullet < m\deg H_j$.  
We have $\vth( \tilde H_j^\bullet )=\zeta^{-\tilde d_j^\bullet}\tilde H_j^\bullet$ and at the same time 
 $\vth( \tilde H_j^\bullet )=\tilde H_j^\bullet$. Thereby
$\tilde d_j^\bullet\le m\deg H_j - m$.  

Suppose now that $\vth(H_j)\ne H_j$, i.e., $\vth(H_j)=\zeta^{r_j} H_j$ with $0 < r_j<m$. 
Here $H_j^\bullet\not\in\gS(\g_0^{\rm ab})$ and $\deg_{\tilde\varphi} \tilde H_j^\bullet \le m\deg H_j - r_j$. 
According to Lemma~\ref{lm:outer-inv}, $\sum_{j=1}^l r_j = \frac{1}{2}m(\rk\g-\rk\g_0)$. Thus 
\begin{multline}
\sum_{i=1}^r \deg_{\tilde\varphi} F_i^\bullet+\sum_{j=1}^r \deg_{\tilde\varphi} \tilde H_j^\bullet+
\sum_{j=r+1}^l \deg_{\tilde\varphi}  H_j^\bullet \le \\ 
\le m\bb(\g_0)+m\bb(\g)-m\rk\g_0- \frac{1}{2}m(\rk\g-\rk\g_0)=\frac{m}{2}(\dim\g+\dim\g_0)=D_{\tilde\vp}. \label{in}
\end{multline}
By Theorem~\ref{thm:kot14}{\sf (i)}, we have the equality in \eqref{in} and 
the polynomials  $F_i, \tilde H_i, H_j$  form a 
 good generating system for $\tilde\vp$. We obtain also 
\begin{equation} \label{degj}
\deg_{\tilde\varphi} H_j^\bullet = m\deg H_j - r_j \ \ \text{ if } \ \ 0<r_j<m.
\end{equation}
  By Theorem~\ref{inf-c2}, the Lie algebra
$\tilde\g$ has the  {\sl codim}--$2$ property. Then Theorem~\ref{thm:kot14}{\sf (iii)} states that 
$\gS(\tilde\g)^{\tilde\g}$ is freely generated by 
$F_i^\bullet, \tilde H_i^\bullet, H_j^\bullet$ with $ 1\le i\le r<j\le l$.
\end{proof} 
 
 Recall that $H_j=\sum_{i\ge 0} H_{j,i}$, where $\varphi_s(H_j)=\sum_{i\ge 0} s^i H_{j,i}$. 
Let  $H_{j,\bullet}$  be the the non-zero bi-homogeneous component of $H_j$ with the minimal $\vp$-degree.
 We regard 
each $H_{j,i}$ as an element of $\gS(\g_{(\infty)})\subset\gS(\tilde\g)$  identifying  $\g$ with the subspace 
\[
\g_{(\infty)}=\g_{m-1}\oplus\ldots\oplus\g_1\oplus\g_0^{\rm ab}\subset\tilde\g.
\] 
Note that $\deg_{\tilde\varphi} H_{j,i}=m\deg H_j - i$. 

\begin{thm} \label{gen-inf}
The algebra  
 $\cz_\infty^{\g_0}$ is freely generated by $F_i^\bullet\in\gS(\g_0^{\rm ab})$ with $1\le i\le r$
and  the set $\{H_{j,\bullet}\mid   r< j\le l\}$, where $H_{j,\bullet}=H_j^\bullet$ for each $j$. 
\end{thm}
\begin{proof}
For any $H_j\in\gS(\g)\subset\gS(\g_0\oplus \g)$, the highest $\tilde\vp$-component $H_j^\bullet\in\gS(\tilde\g)$ is obtained by taking first 
$H_{j,\bullet}\in\gS(\g)$ and then replacing in it each element of $\g_0$ by its copy in
$\g_0^{\rm ab}$. This is exactly how we understand $H_{j,\bullet}\in\gS(\tilde\g)$. 

By the construction and Proposition~\ref{prop:bullet}, 
\[\calH:=\{F_i^\bullet, H_j^\bullet\mid 1\le i\le r,\, r<j\le l\}\subset\cz_\infty^{\g_0}.
\] 
These homogeneous polynomials are  algebraically independent, see Theorem~\ref{ggs-inf}. In view of Theorem~\ref{inf-inv}{\sf (iii)}, we have 
$\trdeg\cz_\infty^{\g_0}=\rk\g=l$. 
By Theorems~\ref{thm:kot14}{\sf (iii)},~\ref{inf-c2},~\ref{ggs-inf}, 
$\tilde\g$ satisfies  the
Kostant regularity criterion and the differentials of the generators $F_i^\bullet, \tilde H_i^\bullet, H_j^\bullet$ with $1\le i\le r <j\le l$ are linearly independent on  a big open 
subset. Thus also
\beq \label{H}
\dim\gJ(\calH)\le \dim\g-2,
\eeq
where $\gJ(\calH)\subset\g^*_{(\infty)}$. 
%
By Theorem~\ref{ppy-max}, 
the subalgebra generated by $\calH$ is algebraically closed in $\gS(\g_{(\infty)})$ and hence 
it coincides with $\cz_\infty^{\g_0}$.
\end{proof}

If $\vartheta$ is inner, then $\cz_\infty=\gS(\g_0)$ \cite{fo} and $\cz_\infty^{\g_0}=\gS(\g_0)^{\g_0}$. One does not need any heavy machinery in that case. If $\vth$ is outer, the situation is different. We have seen this in Theorem~\ref{gen-inf} and 
are going to obtain a description of $\cz_\infty$ next.   

\begin{prop}   \label{zinf-au}
If $\vartheta$ is outer, then a basis $\{y_1,\ldots,y_R\}$ of $\g_0$ together with the lowest $\vp$-components $H_{j,\bullet}$, where 
$r<j\le l$,  
freely  generate $\cz_\infty$. 
\end{prop}
\begin{proof}  
We check that the differentials $\textsl{d}y_i$ and $\textsl{d}H_{j,\bullet}$ of the proposed generators are linearly independent on a big open subset of 
$\g^*$. For $\xi\in\g^*$, set $\bar\xi=\xi|_{\g_0}$. Note that $\bar\xi\in(\g_0^*)_{\sf reg}$ for all $\xi$ in a big open subset of $\g^*$. 
Let $\varpi_\xi\!: T^*_\xi \g^* \to T^*_\xi(G_0{\cdot}\xi)$ be the restriction map. 
Since $H_{j,\bullet}$ is a $G_0$-invariant, we have $\varpi_\xi(\textsl{d}_\xi H_{j,\bullet})=0$ 
for each $j$.  Note that $\ker({\varpi_\xi}|_{\g_0})= \g_0^{\bar\xi}$.

Suppose that $\bar\xi\in(\g_0^*)_{\sf reg}$. 
Then  $\dim\varpi_\xi(\g_0)=\dim\g_0-\rk\g_0$. 
By  the
Kostant regularity criterion~\cite[Theorem~9]{ko63}  applied to $\g_0$, we have 
\[
\left<\textsl{d}_{\xi}F_i^\bullet \mid 1\le i\le r\right>_{\bbk}=\g_0^{\bar\xi}.
\]
Suppose now that $\xi\not\in\gJ(\calH)$. This additional  restriction still leaves us a big open subset of suitable elements, see 
\eqref{H}. Then 
\[
 \dim(\g_0^{\bar\xi} + \left<\textsl{d}_{\xi} H_{j,\bullet}\mid  r<j\le l\right>_{\bbk})=l.
\] 
Thus, on a big open subset, the dimension of 
$V_\xi=\left<\textsl{d}_\xi y_i,\textsl{d}_{\xi} H_{j,\bullet}\mid 1\le i\le R,\,r<j\le l\right>_{\bbk}$ is equal to the sum of 
$\dim\varpi_\xi(V_\xi)=\dim\g_0-\rk\g_0$ 
and $l=\dim\ker({\varpi_\xi}|_{V_\xi})$, i.e.,  
$\dim V_\xi=R+(l-r)$. This number is equal to the number of generators. 
By Theorem~\ref{ppy-max}, 
the subalgebra 
$$
\mathsf{alg}\langle y_i, H_{j,\bullet}\mid 1\le i\le R,\,r<j\le l\rangle
$$ 
is algebraically closed in $\gS(\g_{(\infty)})$. 
Furthermore, 
\[
\mathsf{alg}\langle y_i, H_{j,\bullet}\mid 1\le i\le R,\,r<j\le l\rangle \subset\cz_{\infty} \ \ \text{and} \ \ 
\trdeg\cz_{\infty}\le \ind\g_{(\infty)}=R+l-r,
\]
see Theorem~\ref{thm:ind-inf} for the last equality. Thereby $\mathsf{alg}\langle y_i, H_{j,\bullet}\mid 1\le i\le R,\,r<j\le l\rangle =\cz_{\infty}$.
 \end{proof}

In \cite[Sect.\,5]{fo}, we have considered a larger Poisson-commutative 
 subalgebra 
$\tilde\gZ:=\mathsf{alg}\langle\gZ,\gS(\g_0)^{\g_0}\rangle$. 

\begin{cl}      \label{4.8}
Suppose that $\ind\g_{(0)}=\rk\g$ and that $\gS(\g_{(0)})^{\g_{(0)}}=\mK[H_1^\bullet,\dots,H_l^\bullet]$,
where each $H_j$ is  homogeneous, $\vth(H_j)\in\mK H_j$, and   the polynomials $H_j^\bullet$ are highest $\vp$-components. 

\textbullet\quad  If $\infty\in\BP_{\sf reg}$, then 
$\mathsf{alg}\langle \gZ, \cz_\infty^{\g_0}\rangle=\mathsf{alg}\langle \gZ, \cz_\infty\rangle=\tilde\gZ=\gZ=\mathsf{alg}\!\left<\gZ_{\times},\g_0\right>$; 

\textbullet\quad if $\infty\not\in\BP_{\sf reg}$, then $\mathsf{alg}\langle \gZ, \cz_\infty^{\g_0}\rangle=\tilde\gZ=\mathsf{alg}\langle\gZ_{\times},\gS(\g_0)^{\g_0}\rangle$.  \\
In both cases,  $\tilde\gZ$ is a free (polynomial) algebra. 
\end{cl}
\begin{proof}
Our assumptions on $\gS(\g_{(0)})^{\g_{(0)}}$ imply that $\cz_0\subset\gZ_\times$. 
Furthermore, $\{H_1,\ldots,H_l\}$ is a {\sf g.g.s.} for $\vth$. Then $\gZ_\times$ is a polynomial algebra 
by Theorem~\ref{free-main}. It has $\bb(\g,\vth)$ algebraically independent generators $H_{j,i}$. 
Exactly $r=\rk\gt g_0$ of these generators belong to $\gS(\g_0)$, see Lemmas~\ref{lm:outer-inv},\,\ref{lm-restr}.
If we replace them with algebraically independent generators of $\gS(\g_0)^{\g_0}$, the new algebra
is still polynomial. This finishes the proof in view of Theorem~\ref{gen-inf}. 
\end{proof}

The first statement of
Corollary~\ref{4.8} generalises Corollary~\ref{cor2:main4} to outer automorphisms.

\end{document}